\documentclass[12pt,notitlepage]{amsart}
\usepackage{latexsym,amsfonts,amssymb,amsmath,amsthm} %{pstricks,pst-node,epsf}
\usepackage{color}
\newtheorem{cor}{Corollary}
\newtheorem{thm}{Theorem}
\newtheorem{lem}{Lemma}

\newtheorem{conj}{Conjecture}
\newtheorem{defi}{Definition}

\theoremstyle{remark}
\newtheorem{rmk}{Remark}

\theoremstyle{plain}

\theoremstyle{remark}

\numberwithin{equation}{section}
\def\sumstar{\operatornamewithlimits{\sum\nolimits^*}}

\begin{document}

%\mbox{}

%\begin{flushright}
%\textbf{Julio Andrade}\\
%ICERM--Brown University\\
%Providence, RI--02903\\
%USA\\
%\vspace{0.5cm}
%j.c.andrade.math@gmail.com\\
%\vspace{0.5cm}
%04/18/2013\\ \\
%\end{flushright}
%\vspace{2.0cm}

%Dear Journal of Number Theory Editor, 
%\vspace{0.5cm}

%We would like to submit the attached paper, entitled ``Conjectures for the Integral Moments and Ratios of $L$--functions over function fields" to be considered for publication in the Journal of Number Theory. The paper deals with the famous Integral Moments and Ratios conjectures of $L$--functions and we present Conjectures for the Integral Moments and Ratios of quadratic Dirichlet $L$--functions for the rational function field. Attached below is the file containing our paper. Please, let us know of any news or update concerning the paper. Thank you!

%\begin{center}
%\vspace{2.0cm}

%Sincerely,

%\vspace{1.0cm}

%Julio Andrade.
%\end{center}

\thispagestyle{empty}
\newpage

\title[Integral Moments and Ratios of $L$--functions over function fields]{Conjectures for the Integral Moments and Ratios of $L$--functions over function fields}

\author{J. C. Andrade}
\address{School of Mathematics, University of Bristol, Bristol BS8 1TW, UK}
\email{j.c.andrade@bristol.ac.uk}
\thanks{JCA was supported by an Overseas Research Scholarship and an University of Bristol Research Scholarship.}

\author{J. P. Keating}
\address{School of Mathematics, University of Bristol, Bristol BS8 1TW, UK}
\email{j.p.keating@bristol.ac.uk}
\thanks{JPK is sponsored by the Air Force Office of Scientific Research, Air Force Material Command, USAF, under grant number FA8655-10-1-3088. The U.S. Government is authorized to reproduce and distribute reprints for Governmental purpose notwithstanding any copyright notation thereon. JPK is also grateful to the Leverhulme Trust for support.}


\begin{abstract}
We extend to the function field setting the heuristic previously developed, by Conrey, Farmer, Keating, Rubinstein and Snaith, for the integral moments and ratios of $L$--functions defined over number fields. Specifically, we give a heuristic for the moments and ratios of a family of $L$-functions associated with hyperelliptic curves of genus $g$ over a fixed finite field $\mathbb{F}_{q}$ in the limit as $g\rightarrow\infty$. Like in the number field case, there is a striking resemblance to the corresponding formulae for the characteristic polynomials of random matrices. As an application, we calculate the one--level density for the zeros of these $L$-functions. 
\end{abstract}

\maketitle

\tableofcontents

\section{Introduction}
\subsection{Moments of the Riemann zeta function}
There has been a long-standing interest in the mean values of families of $L$--functions. In the case of the Riemann zeta function, the goal is to determine of the asymptotic behaviour of
\begin{equation}
M_{k}(T)=\int_{0}^{T}|\zeta(\tfrac{1}{2}+it)|^{2k}dt,
\end{equation}
as $T\rightarrow\infty$. Hardy and Littlewood \cite{HL} established in 1918 that
\begin{equation}
M_{1}(T)\sim T\log T,
\end{equation}
and in 1926 Ingham \cite{I} showed that
\begin{equation}
M_{2}(T)\sim \frac{1}{2\pi^{2}}T\log^{4}T.
\end{equation} 
For other values of $k$ the problem is still open. It is believed that for a given $k$
\begin{equation}
\label{eq:1.4}
M_{k}(T)\sim C_{k}T(\log T)^{k^{2}},
\end{equation}
for a positive constant $C_{k}$. Conrey and Ghosh \cite{CG} presented \eqref{eq:1.4} in a more explicit form, in which 
\begin{equation}
\label{eq:1.5}
C_{k}=\frac{a_{k}g_{k}}{\Gamma(k^{2}+1)},
\end{equation} 
where
\begin{equation}
\label{eq:1.6}
a_{k}=\prod_{p}\left[\left(1-\frac{1}{p}\right)^{k^{2}}\sum_{m\geq0}\frac{d_{k}(m)^{2}}{p^{m}}\right],
\end{equation}
and $g_{k}$ should be an integer. The classical results of Hardy--Littlewood and Ingham imply that $g_{1}=1$ and $g_{2}=2$. Based on an analogy with the characteristic polynomials of random matrices, Keating and Snaith \cite{KeS1} conjectured a precise value for $C_{k}$ for $\mathfrak{R}(k)>-\tfrac{1}{2}$.
\begin{conj}[Keating--Snaith]
\label{conj:conj1}
For $k$ fixed and $\mathfrak{R}(k)>-\tfrac{1}{2}$,
\begin{equation}
M_{k}(T)=\int_{0}^{T}|\zeta(\tfrac{1}{2}+it)|^{2k}dt\sim\frac{a_{k}g_{k}}{(k^{2})!}T(\log T)^{k^{2}},
\end{equation}
as $T\rightarrow\infty$, where $a_{k}$ is the arithmetic factor given by \eqref{eq:1.6} and the random matrix theory factor $g_{k}$ is given by
\begin{equation}
g_{k}:=\lim_{N\rightarrow\infty}\frac{\Gamma(k^{2}+1)}{N^{k^{2}}}\int_{U(N)}|\Lambda_{A}(e^{0})|^{2k}dA=(k^{2})!\frac{G^{2}(1+k)}{G(1+2k)},
\end{equation}
where $\Lambda_{A}$ is the characteristic polynomial of a unitary $N\times N$ matrix $A$, $dA$ denotes the Haar measure on $U(N)$, and $G(z)$ denotes the Barnes $G$--function \cite{BA}.
\end{conj}
\begin{rmk}
For $k\in\mathbb{N}$
\begin{equation}
(k^{2})!\frac{G^{2}(1+k)}{G(1+2k)}=(k^{2})!\prod_{j=0}^{k-1}\frac{j!}{(j+k)!},
\end{equation}
is an integer.
\end{rmk}

The separation into arithmetic and random matrix factors, $a_{k}$ and $g_{k}$ respectively, in \eqref{eq:1.5} is explained by a hybrid product formula for $\zeta(s)$ that includes both the primes and the zeros \cite{GHK}. %For $k\in\mathbb{N}$
%\begin{equation}
%(k^{2})!\frac{G^{2}(1+k)}{G(1+2k)}=(k^{2})!\prod_{j=0}^{k-1}\frac{j!}{(j+k)!},
%\end{equation}
%is an integer.

\subsection{Mean values of $L$--functions}

For the family of quadratic Dirichlet $L$--functions $L(s,\chi_{d})$, with $\chi_{d}$ a real primitive Dirichlet character modulo $d$ given by the Kronecker symbol $\chi_{d}(n)=\left(\tfrac{d}{n}\right)$, the goal is to determine the asymptotic behaviour of

\begin{equation}
\sum_{0<d\leq D}L(\tfrac{1}{2},\chi_{d})^{k}
\end{equation} 
as $D\rightarrow\infty$. Jutila \cite{J} proved in 1981 that

\begin{multline}\label{eq:jutilaf}
\sum_{0<d\leq D}L(\tfrac{1}{2},\chi_{d})=\frac{P(1)}{4\zeta(2)}D\left\{\log(D/\pi)+\frac{\Gamma^{'}}{\Gamma}(1/4)+4\gamma-1+4\frac{P^{'}}{P}(1)\right\}+O(D^{3/4+\varepsilon})
\end{multline}
where
\begin{equation}
P(s)=\prod_{p}\left(1-\frac{1}{(p+1)p^{s}}\right),
\end{equation}
and
\begin{equation}
\label{eq:jutilas}
\sum_{0<d\leq D}L(\tfrac{1}{2},\chi_{d})^{2}=\frac{c}{\zeta(2)}D\log^{3}D+O(D(\log D)^{5/2+\varepsilon})
\end{equation}
with
\begin{equation}
c=\frac{1}{48}\prod_{p}\left(1-\frac{4p^{2}-3p+1}{p^{4}+p^{3}}\right).
\end{equation}
Restricting $d$ to be odd, square--free and positive, so that $\chi_{8d}$ are real, primitive characters with conductor $8d$ and with $\chi_{8d}(-1)=1$, Soundararajan \cite{S} proved that
\begin{multline}
\label{eq:sound}
\frac{1}{D^{*}}\sideset{}{^*}\sum_{0<d\leq D}L(\tfrac{1}{2},\chi_{8d})^{3}\\
\sim\frac{1}{184320}\prod_{p\geq3}\left(1-\frac{12p^{5}-23p^{4}+23p^{3}-15p^{2}+6p-1}{p^{6}(p+1)}\right)(\log D)^{6},
\end{multline}
where the sum $\sum\nolimits^{*}$ runs over the restricted set, and $D^{*}$ is the number of such $d$ in $(0,D]$. For other values of $k$ the problem is still open. 

Extending their approach to the zeta--function moments, Keating and Snaith \cite{KeS2} established the following conjecture for the mean value of quadratic Dirichlet $L$--functions.

\begin{conj}[Keating--Snaith]
\label{conj:conj2}
For $k$ fixed with $\mathfrak{R}(k)\geq0$, as $D\rightarrow\infty$
\begin{equation}
\frac{1}{D^{*}}\sideset{}{^{*}}\sum_{0<d\leq D}L(\tfrac{1}{2},\chi_{8d})^{k}\sim a_{k,Sp}\frac{G(k+1)\sqrt{\Gamma(k+1)}}{\sqrt{G(2k+1)\Gamma(2k+1)}}(\log D)^{k(k+1)/2}
\end{equation}
where
$$a_{k,Sp}=2^{-k(k+2)/2}\prod_{p\geq3}\frac{(1-\frac{1}{p})^{k(k+1)/2}}{1+\frac{1}{p}}\left(\frac{(1-\frac{1}{\sqrt{p}})^{-k}+(1+\frac{1}{\sqrt{p}})^{-k}}{2}+\frac{1}{p}\right)$$
and $G(z)$ is Barnes' $G$--function.
\end{conj}
This conjecture is also in agreement with previous results from Jutila (equations \eqref{eq:jutilaf} and \eqref{eq:jutilas}), Soundararajan \eqref{eq:sound} and with the conjectures given by Conrey and Farmer in \cite{CF}. The separation into arithmetical and random matrix factors is again explained by  a hybrid product formula \cite{BK}.

\subsection{Integral Moments of $L$--functions}

Conrey, Farmer, Keating, Rubinstein and Snaith \cite{CFKRS, CFKRS1} developed a ``recipe", making use of heuristic arguments, for a sharpened form of the Conjectures \ref{conj:conj1} and \ref{conj:conj2} for integral $k$. Specifically, they gave conjectures beyond the leading order asymptotics to include all the principal lower order terms. For example, their conjecture for quadratic Dirichlet $L$--functions (see \cite{CFKRS}) takes the following form.

\begin{conj}[Conrey, Farmer, Keating, Rubinstein, Snaith]
\label{conj:conj3}
Let $X_{d}(s)=|d|^{1/2-s}X(s,a)$ where $a=0$ if $d>0$ and $a=1$ if $d<0$, and
\begin{equation}
X(s,a)=\pi^{s-1/2}\Gamma\left(\frac{1+a-s}{2}\right)\Big/\Gamma\left(\frac{s+a}{2}\right).
\end{equation}
That is, $X_{d}(s)$ is the factor in the functional equation for the quadratic Dirichlet $L$--function
\begin{equation}
L(s,\chi_{d})=\varepsilon_{d}X_{d}(s)L(1-s,\chi_{d}).
\end{equation}
Summing over fundamental discriminants $d$ 
\begin{equation}
\sumstar_{\!d}L(\tfrac12 ,\chi_d)^k= \sumstar_{\!d}\,Q_k(\log
{|d|})(1+o(1))
\end{equation}
where $Q_k$ is the polynomial of degree $k(k+1)/2$ given by
the $k$-fold residue
\begin{eqnarray}
Q_k(x) & = & \frac{(-1)^{k(k-1)/2}2^k}{k!}
\frac{1}{(2\pi i)^{k}}
\oint \cdots \oint
\frac{G(z_1, \dots,z_{k})\Delta(z_1^2,\dots,z_{k}^2)^2}
{\prod_{j=1}^{k} z_j^{2k-1}}\nonumber\\
& & \ \ \ \times e^{\tfrac x2 \sum_{j=1}^{k}z_j}\,dz_1\dots dz_{k} ,
\end{eqnarray}
with
\begin{equation}
G(z_1,\dots,z_k)=A_k(z_1,\dots,z_k) 
\prod_{j=1}^k X(\tfrac12+z_j,a)^{-\frac12}
%\left(
%    \frac{\Gamma(\frac 34+\frac {z_j}2) 2^{z_j}}
%         {\Gamma(\frac 34- \frac {z_j}2)}
%\right)^{\tfrac 12} 
\prod_{1\le i\le j\le k}\zeta(1+z_i+z_j),
\end{equation}
$\Delta(z_{1},\ldots,z_{k})$ the Vandermonde determinant given by
\begin{equation}
\label{eq:vandermonde}
\Delta(z_{1},\ldots,z_{k})=\prod_{1\leq i<j\leq k}(z_{j}-z_{i}),
\end{equation}
and $A_k$ is the Euler product, absolutely convergent for
$|\Re z_j|<\frac12 $, defined by
\begin{eqnarray}
A_k(z_1,\dots,z_k) & = & \prod_p \prod_{1\le i \le j \le k}
\left(1-\frac{1}{p^{1+z_i+z_j}}\right) \nonumber \\&   &\times \left(\frac
12 \left(\prod_{j=1}^k\left( 1-\frac{1}{p^{\frac 12+z_j}}\right)^{-1} +
\prod_{j=1}^k\left(1+\frac{1}{p^{\frac12+z_j}}\right)^{-1}
\right)+\frac 1p \right)\nonumber\\
& & \times \left( 1+ \frac{1}{p}\right)^{-1}. 
\end{eqnarray}
\end{conj}    

\begin{rmk}
Conjecture \ref{conj:conj3} was originally stated with error term $O(|d|^{-\tfrac{1}{2}+\varepsilon})$, but it appears there are extraneous lower order terms, as firstly pointed out by Diaconu, Goldfeld and Hoffstein \cite{DGH}, with the remainder term being larger for $k\geq3$. This is supported numerically by the computations of Alderson and Rubinstein \cite{AR}. We have therefore limited ourselves to restating it with an error that is simply $o(1)$.
\end{rmk}

Conjecture \ref{conj:conj3} is closely analogous to exact formulae for the moments of the characteristic polynomials of random matrices \cite{CFKRS, CFKRS2}. By different methods (Multiple Dirichlet Series Techniques) Diaconu, Goldfeld and Hoffstein also have obtained a conjectural formula for the moments of quadratic Dirichlet $L$--functions. 
%It remains to be verified that these conjectures coincide. 

Recently Bui and Heath--Brown \cite{BH} showed that for $q,T\geq2$
 %The recent breakthrough in the mean values of a general Dirichlet $L$--function are due to Bui and Heath--Brown \cite{BH} regarding the fourth moment and due Conrey, Iwaniec and Soundararajan \cite{CIS} about the sixth moment. Bui and Heath--Brown showed that for $q,T\geq2$

\begin{multline}
\label{BH}
\sideset{}{^*}\sum_{\chi\bmod q}\int_{0}^{T}|L(\tfrac{1}{2}+it,\chi)|^{4}dt\\
=\left(1+O\left(\frac{\omega(q)}{\log q}\sqrt{\frac{q}{\phi(q)}}\right)\right)\frac{\phi^{*}(q)T}{2\pi^{2}}\prod_{p\mid q}\frac{(1-p^{-1})^{3}}{(1+p^{-1})}(\log qT)^{4}+O(qT(\log qT)^{7/2}),
\end{multline}
where the sum is over all primitive Dirichlet character $\chi$ modulo $q$, $\omega(q)$ is the number of distinct prime factors of $q$, and $\phi^{*}(q)$ is the number of primitive Dirichlet character, and Conrey, Iwaniec and Soundararajan \cite{CIS} obtained the following asymptotic formula for the sixth moment:

\begin{multline}
\label{CIS}
\sum_{q\leq Q} \ \ \sideset{}{^*}\sum_{\chi\bmod q}\int_{-\infty}^{\infty}|\Lambda(\tfrac{1}{2}+iy,\chi)|^{6}dy \\
\sim42a_{3}\sum_{q\leq Q}\prod_{p\mid q}\frac{(1-\tfrac{1}{p})^{5}}{(1+\tfrac{4}{p}+\tfrac{1}{p^{2}})}\varphi^{*}(q)\frac{(\log q)^{9}}{9!}\int_{-\infty}^{\infty}\left|\Gamma\left(\frac{1/2+iy}{2}\right)\right|^{6}dy,
\end{multline}
where $\chi$ is a primitive even Dirichlet character modulo $q$, $a_{3}$ is a certain product over primes, $\varphi^{*}(q)$ is the number of even primitive Dirichlet characters and
\begin{equation}
\Lambda(\tfrac{1}{2}+s,\chi):=\left(\frac{q}{\pi}\right)^{s/2}\Gamma\left(\frac{1}{4}+\frac{s}{2}\right)L(\tfrac{1}{2}+s,\chi).
\end{equation}
Both \eqref{BH} and \eqref{CIS} are consistent with our general conjectural understanding of moments.

\subsection{Ratios Conjectures}
Conrey, Farmer and Zirnbauer \cite{CFZ} presented a generalization of the heuristic arguments used in \cite{CFKRS} leading to conjectures for the ratios of products of $L$--functions. These conjectures are very useful, for example it is possible to obtain from them all $n$--level correlations of zeros with lower order terms \cite{CS1} (c.f. also \cite{BKe,BKe1,BKe2}), averages of mollified $L$--functions, discrete moments of Riemann zeta function and non--vanishing results for various familes of $L$--functions. For more details about these applications see \cite{CS}.

We will quote in this paper the ratios conjecture for quadratic Dirichlet $L$--functions from \cite{CFZ}, since we will use it to compare with the results presented in the section 3.

\begin{conj}[Conrey, Farmer, Zirnbauer]
\label{conj:conj4}
Let $\mathcal{D}^{+}=\{L(s,\chi_{d}):d>0\}$ to be the symplectic family of $L$--functions associated with the quadratic character $\chi_{d}$, and suppose that the real parts of $\alpha_{k}$ and $\gamma_{q}$ are positive. Then
\begin{multline}
\sum_{0<d\le X}\frac{\prod_{k=1}^K L(1/2+\alpha_k,\chi_d)}
{\prod_{m=1}^Q L(1/2+\gamma_m,\chi_d)}=\sum_{0<d\le X} \sum_{\epsilon \in \{-1,1\}^K}
\bigg(\frac{|d|}{\pi}\bigg)^{\frac 12\sum_{k=1}^K(\epsilon_k\alpha_k-\alpha_k)}\\
\times\prod_{k=1}^K g_+\left(\frac 12+\frac{\alpha_k-\epsilon_k \alpha_k }{2}\right)
Y_SA_\mathcal D(\epsilon_1 \alpha_{1},\dots, \epsilon_K  \alpha_{K};
 \gamma )+o(X) .
\end{multline}
where
\begin{equation}
g_{+}(s)=\frac{\Gamma\left(\frac{1-s}{2}\right)}{\Gamma\left(\frac{s}{2}\right)},
\end{equation}
\begin{equation}
Y_{S}(\alpha;\gamma):=\frac{\prod_{j\leq k\leq K}\zeta(1+\alpha_{j}+\alpha_{k})\prod_{q<r\leq Q}\zeta(1+\gamma_{q}+\gamma_{r})}{\prod_{k=1}^{K}\prod_{q=1}^{Q}\zeta(1+\alpha_{k}+\gamma_{q})},
\end{equation}
and
\begin{align}
A_\mathcal D(\alpha,\gamma)
 =&\prod_p \frac{\prod_{  j\le k\le K}  (1-1/p^{1+\alpha_j+\alpha_k})
\prod_{ q< r\le Q}(1-1/p^
{1+\gamma_q+\gamma_r})}{\prod_{k=1}^K\prod_{q=1}^Q(1-1/p^{1+\alpha_k+\gamma_q})
  }\cr
& \qquad \times
\bigg(1+(1+\tfrac 1p)^{-1}
\sum_{0< \sum_k a_k +\sum_q c_q  \mbox{ is even}}
\frac{\prod_q \mu(p^{c_q})}
{  p ^{\sum_k a_k(1/2+\alpha_k) +\sum_q c_q(1/2+\gamma_q)}}\bigg) .
\end{align}
\end{conj}

\subsection{Structure of the Paper}

In this paper we develop the function field analogues of conjectures \ref{conj:conj3} and \ref{conj:conj4} for the family of quadratic Dirichlet $L$--functions associated with hyperelliptic curves of genus $g$ over a fixed finite field $\mathbb{F}_{q}$. In section 2, we present a background on $L$--functions over function fields and how to average in this context. In section 3, we present our main results: the integral moments conjecture and ratios conjectures for $L$--functions in the hyperelliptic ensemble. In section 4, we outline the adaptation of the recipe of \cite{CFKRS} for the function field setting. In section 5 we use the integral moments conjecture over function fields to compare with the main theorem established in \cite{AK} when $k=1$ and to conjecture precise values of moments for the case $k=2$ and $k=3$ in this setting. In section 6, we adapt the recipe of Conrey, Farmer and Zirnbauer \cite{CFZ} for the same family of $L$--functions over function fields and again we compare our conjecture with the original ratios conjecture for a symplectic family. In section 7 we use the the ratios conjecture to compute the one--level density of the zeros of the same family of $L$--functions.

\section{Some Basic facts about $L$--functions in function fields}

We begin by fixing a finite field $\mathbb{F}_{q}$ of odd cardinality and letting $A=\mathbb{F}_{q}[x]$ be the polynomial ring over $\mathbb{F}_{q}$ in the variable $x$. We will denote by $C$ any smooth, projective, geometrically connected curve of genus $g\geq1$ defined over the finite field $\mathbb{F}_{q}$. The zeta function of the curve $C$, first introduced by Artin \cite{A}, is defined as

\begin{equation}
Z_{C}(u):=\exp\left(\sum_{n=1}^{\infty}N_{n}(C)\frac{u^{n}}{n}\right), \ \ \ \ \ |u|<1/q
\end{equation}
where $N_{n}(C):=\mathrm{Card}(C(\mathbb{F}_{q}))$ is the number of points on $C$ with coordinates in a field extension $\mathbb{F}_{q^{n}}$ of $\mathbb{F}_{q}$ of degree $n\geq1$. Weil \cite{W} showed that the zeta function associated to $C$ is a rational function of the form

\begin{equation}\label{eq:zetaC}
Z_{C}(u)=\frac{P_{C}(u)}{(1-u)(1-qu)},
\end{equation}
where $P_{C}(u)\in\mathbb{Z}[u]$ is a polynomial of degree $2g$ with $P_{C}(0)=1$ that satisfies the functional equation

\begin{equation}\label{eq:funceq}
P_{C}(u)=(qu^{2})^{g}P_{C}\left(\frac{1}{qu}\right).
\end{equation}
By the Riemann Hypothesis for curves over finite fields, also proved by Weil \cite{W}, one knows that the zeros of $P_{C}(u)$ all lie on the circle $|u|=q^{-1/2}$, i.e.,
\begin{equation}
P_{C}(u)=\prod_{j=1}^{2g}(1-\alpha_{j}u), \ \ \ \ \ \mathrm{with} \ \ |\alpha_{j}|=\sqrt{q} \ \ \mathrm{for \ all}\ j.
\end{equation}

\subsection{Background on $\mathbb{F}_{q}[x]$}
The norm of a polynomial $f\in\mathbb{F}_{q}[x]$ is, for $f\neq0$, defined to be $|f|:=q^{\mathrm{deg}f}$ and if $f=0$, $|f|=0$. A monic irreducible polynomial is called a ``prime" polynomial.

The zeta function of $A=\mathbb{F}_{q}[x]$, denoted by $\zeta_{A}(s)$, is defined by the infinite series

\begin{equation}\label{eq:zetaA}
\zeta_{A}(s):=\sum_{\substack{f\in A \\ f \ \mathrm{monic}}}\frac{1}{|f|^{s}}=\prod_{\substack{P \ \mathrm{monic} \\ \mathrm{irreducible}}}\left(1-|P|^{-s}\right)^{-1}, \ \ \ \ \ \ \mathfrak{R}(s)>1
\end{equation}
which is
\begin{equation}\label{eq:zetaA1}
\zeta_{A}(s)=\frac{1}{1-q^{1-s}}.
\end{equation}
The analogue of the Mobius function $\mu(f)$ for $A=\mathbb{F}_{q}[x]$ is defined as follows:
\begin{equation}\label{eq:3.3}
\mu(f)=\left\{
\begin{array}{rcl}
(-1)^{t}, & f=\alpha P_{1}P_{2}\ldots P_{t},\\
0, & \mathrm{otherwise},\\
\end{array}
\right.
\end{equation}
where each $P_{j}$ is a distinct monic irreducible.

\subsection{Quadratic Characters and the Corresponding $L$--functions}
Assume from now on that $q$ is odd and let $P(x)\in\mathbb{F}_{q}[x]$ be an irreducible polynomial. 
%(see Lemma \ref{lem:lem4.3} in Section \ref{sec4} for an explanation for this restriction)
%Then \cite[Proposition 1.10]{Ro} if $f\in A$ and $P\nmid f$ we know that the congruence $x^{d}\equiv f\pmod P$ is solvable if and only if
%$$f^{\frac{|P|-1}{d}}\equiv1\pmod P,$$
%where $d$ is a divisor of $q-1$ and for our purpose $d=2$. So if $P\nmid f$, let $(f/P)$ be the unique element of $\mathbb{F}^{*}$ such that
%$$f^{\frac{|P|-1}{2}}\equiv\left(\frac{f}{P}\right)\pmod P$$  
%and if $P\mid f$ we define $(f/P)=0$.

In this way we can define the quadratic residue symbol $(f/P)\in\{\pm1\}$ for $f$ coprime to $P$ by
\begin{equation}
\left(\frac{f}{P}\right)\equiv f^{(|P|-1)/2}\pmod P.
\end{equation}
We can also define the Jacobi symbol $(f/Q)$ for arbitrary monic $Q$: let $f$ be coprime to $Q$ and $Q=\alpha P_{1}^{e_{1}}P_{2}^{e_{2}}\ldots P_{s}^{e_{s}}$, then
\begin{equation}
\left(\frac{f}{Q}\right)=\prod_{j=1}^{s}\left(\frac{f}{P_{j}}\right)^{e_{j}};
\end{equation}
if $f,Q$ are not coprime we set $(f/Q)=0$ and if $\alpha\in\mathbb{F}_{q}^{*}$ is a scalar then

\begin{equation}
\left(\frac{\alpha}{Q}\right)=\alpha^{((q-1)/2)\mathrm{deg}Q}.
\end{equation}

%The analogue of the quadratic reciprocity law for function fields is

%\begin{thm}[Quadratic reciprocity]

%Let $A,B\in\mathbb{F}_{q}[x]$ be relatively prime and $A\neq0$ and $B\neq0$. Then,

%$$\left(\frac{A}{B}\right)=\left(\frac{B}{A}\right)(-1)^{((q-1)/2)\mathrm{deg}(A)\mathrm{deg}(B)}=\left(\frac{B}{A}\right)(-1)^{((|A|-1)/2)((|B|-1)/2)}$$

%\end{thm}
Now we present the definition of quadratic characters for $\mathbb{F}_{q}[x]$.
\begin{defi}
Let $D\in\mathbb{F}_{q}[x]$ be square-free.  We define the \textit{quadratic character} $\chi_{D}$ using the quadratic residue symbol for $\mathbb{F}_{q}[x]$ by
\begin{equation}
\chi_{D}(f)=\left(\frac{D}{f}\right).
\end{equation}
So, if $P\in A$ is monic irreducible we have
\begin{equation}
\chi_{D}(P)=\left\{
\begin{array}{cl}
0, & \mathrm{if}\ P\mid D,\\
1, & \mathrm{if}\ P\not{|} D \ \mathrm{and} \ D \ \mathrm{is \ a \ square \ modulo} \ P,\\
-1, & \mathrm{if}\ P\not{|} D \ \mathrm{and} \ D \ \mathrm{is \ a \ non\ square \ modulo} \ P.\\
\end{array}
\right.
\end{equation}
\end{defi}
We define the $L$--function corresponding to the quadratic character $\chi_{D}$ by
\begin{equation}
\mathcal{L}(u,\chi_{D}):=\prod_{\substack{P \ \mathrm{monic}\\ \mathrm{irreducible}}}(1-\chi_{D}(P)u^{\mathrm{deg}P})^{-1}, \ \ \ \ \ |u|<1/q
\end{equation}
where $u=q^{-s}$. The $L$--function above can also be expressed as an infinite series in the usual way:
%and the product is over all monic irreducible (prime) polynomials $P$. 
\begin{equation}\label{eq:3.8}
\mathcal{L}(u,\chi_{D})=\sum_{\substack{f\in A \\ f \ \mathrm{monic}}}\chi_{D}(f)u^{\mathrm{deg}f}=L(s,\chi_{D})=\sum_{\substack{f\in A \\ f \ \mathrm{monic}}}\frac{\chi_{D}(f)}{|f|^{s}}.
\end{equation}
We can write \eqref{eq:3.8} as
\begin{equation}\label{eq:3.9}
\mathcal{L}(u,\chi_{D})=\sum_{n\geq0}\sum_{\substack{\mathrm{deg}(f)=n \\ f \ \mathrm{monic}}}\chi_{D}(f)u^{n}.
\end{equation}
If we denote
\begin{equation}A_{D}(n):=\sum_{\substack{f \ \mathrm{monic} \\ \mathrm{deg}(f)=n}}\chi_{D}(f),
\end{equation}
we can write \eqref{eq:3.9} as
\begin{equation}
\sum_{n\geq0}A_{D}(n)u^{n},
\end{equation}
and by \cite[Propostion 4.3]{Ro}, if $D$ is a non--square polynomial of positive degree, then $A_{D}(n)=0$ for $n\geq\mathrm{deg}(D)$.  So in this case the $L$--function is in fact a polynomial of degree at most $\mathrm{deg}(D)-1$.

Assuming the primitivity condition that $D$ is a square--free monic polynomial of positive degree and following the arguments presented in \cite{Ru} we have that $\mathcal{L}(u,\chi_{D})$ has a ``trivial" zero at $u=1$ if and only if $\mathrm{deg}(D)$ is even, which enables us to define the ``completed" $L$--function
\begin{equation}\label{eq:3.11}
\mathcal{L}(u,\chi_{D})=(1-u)^{\lambda}\mathcal{L}^{*}(u,\chi_{D}), \ \ \ \ \ \lambda=\left\{
\begin{array}{rcl}
1, & \mathrm{deg}(D) \ \mathrm{even},\\
0, & \mathrm{deg}(D) \ \mathrm{odd},\\
\end{array}
\right.
\end{equation}
where $\mathcal{L}^{*}(u,\chi_{D})$ is a polynomial of even degree
\begin{equation}
2\delta=\mathrm{deg}(D)-1-\lambda
\end{equation}
satisfying the functional equation
\begin{equation}
\mathcal{L}^{*}(u,\chi_{D})=(qu^{2})^{\delta}\mathcal{L}^{*}(1/qu,\chi_{D}).
\end{equation}

By \cite[Proposition 14.6 and 17.7]{Ro}, $\mathcal{L}^{*}(u,\chi_{D})$ is the Artin $L$--function corresponding to the unique nontrivial quadratic character of $\mathbb{F}_{q}(x)(\sqrt{D(x)})$.  The fact that is important for this paper is that the numerator $P_{C}(u)$ of the zeta-function of the hyperelliptic curve $y^{2}=D(x)$ coincides with the completed Dirichlet $L$--function $\mathcal{L}^{*}(u,\chi_{D})$ associated with the quadratic character $\chi_{D}$, as was found in Artin's thesis. So we can write $\mathcal{L}^{*}(u,\chi_{D})$ as
\begin{equation}
\mathcal{L}^{*}(u,\chi_{D})=\sum_{n=0}^{2\delta}A^{*}_{D}(n)u^{n},
\end{equation}
where $A_{D}^{*}(0)=1$ and $A_{D}^{*}(2\delta)=q^{\delta}$.

For $D$ monic, square-free, and of positive degree, the zeta function \eqref{eq:zetaC} of the hyperelliptic curve $y^{2}=D(x)$ is
\begin{equation}\label{eq:3.13}
Z_{C_{D}}(u)=\frac{\mathcal{L}^{*}(u,\chi_{D})}{(1-u)(1-qu)}.
\end{equation}
Note that,

\begin{equation}
L(s,\chi_{D})=\mathcal{L}(u,\chi_{D}), \ \ \ \ \ \ \mathrm{where} \ \ u=q^{-s}
\end{equation}
as $\mathrm{deg}(D)$ is odd. 

\subsection{The Hyperelliptic Ensemble $\mathcal{H}_{2g+1,q}$}
Let $\mathcal{H}_{d}$ be the set of square--free monic polynomials of degree $d$ in $\mathbb{F}_{q}[x]$. The cardinality of $\mathcal{H}_{d}$ is
\begin{equation}
\#\mathcal{H}_{d}=\left\{
\begin{array}{lcl}
(1-1/q)q^{d}, & d\geq2,\\
q, & d=1.\\
\end{array}
\right.
\end{equation}
(This can be proved using
\begin{equation}
\sum_{d>0}\frac{\#\mathcal{H}_{d}}{q^{ds}}=\sum_{\substack{f \ \mathrm{monic} \\ \mathrm{squarefree}}}|f|^{-s}=\frac{\zeta_{A}(s)}{\zeta_{A}(2s)}
\end{equation}
and \eqref{eq:zetaA1}  \cite[Proposition 2.3]{Ro}). In particular, for $D\in\mathcal{H}_{2g+1,q}$ and $g\geq1$ we have,
\begin{equation}\label{eq:3.17}
\#\mathcal{H}_{2g+1,q}=(q-1)q^{2g}=\frac{|D|}{\zeta_{A}(2)}.
\end{equation}
We can treat $\mathcal{H}_{2g+1,q}$ as a probability space (ensemble) with uniform probability measure. Thus the expected value of any continuous function $F$ on $\mathcal{H}_{2g+1,q}$ is defined as
\begin{equation}
\left\langle F(D)\right\rangle:=\frac{1}{\#\mathcal{H}_{2g+1,q}}\sum_{D\in\mathcal{H}_{2g+1,q}}F(D).
\end{equation}

Using the Mobius function $\mu$ of $\mathbb{F}_{q}[x]$ defined in \eqref{eq:3.3} we can sieve out the square-free polynomials, since
\begin{equation}
\sum_{A^{2}|D}\mu(A)=\left\{
\begin{array}{rcl}
1, & D \ \mathrm{square \ free},\\
0, & \mathrm{otherwise}.\\
\end{array}
\right.
\end{equation}
In this way we can write the expected value of any function $F$ as
%Thus we may write the expected value as
\begin{eqnarray}\label{eq:3.20}
\left\langle F(D)\right\rangle &=& \frac{1}{\#\mathcal{H}_{2g+1,q}}\sum_{\substack{D \ \mathrm{monic} \\ \mathrm{deg}(D)=2g+1}}\sum_{A^{2}\mid D}\mu(A)F(D)\\
& = & \frac{1}{(q-1)q^{2g}}\sum_{2\alpha+\beta=2g+1}\sum_{\substack{B \ \mathrm{monic} \\ \mathrm{deg}B=\beta}}\sum_{\substack{A \ \mathrm{monic} \\ \mathrm{deg}A=\alpha}}\mu(A)F(A^{2}B).\nonumber
\end{eqnarray}

\section{Statement of the Main Results}

We now present the main conjectures that will be motivated by extending the recipe of \cite{CFKRS} to the function field setting.

\begin{conj}
\label{thm:myconjecture}
Suppose that $q$ odd is the fixed cardinality of the finite field $\mathbb{F}_{q}$ and let $\mathcal{X}_{D}(s)=|D|^{1/2-s}X(s)$ and 
\begin{equation}
X(s)=q^{-1/2+s}.
\end{equation}
That is, $\mathcal{X}_{D}(s)$ is the factor in the functional equation
\begin{equation}
\label{eq:functionalequation} 
L(s,\chi_{D})=\mathcal{X}_{D}(s)L(1-s,\chi_{D}).
\end{equation}
Summing over fundamental discriminants $D\in\mathcal{H}_{2g+1,q}$ we have
\begin{equation}
\label{eq:eq5.1.11}
\sum_{D\in\mathcal{H}_{2g+1,q}}L(\tfrac12 ,\chi_D)^k= \sum_{D\in\mathcal{H}_{2g+1,q}}\,Q_k(\log_{q}
{|D|})(1+o(1)) 
\end{equation}
where $Q_{k}$ is the polynomial of degree $k(k+1)/2$ given by the $k$--fold residue 
\begin{eqnarray}
Q_k(x) & = & \frac{(-1)^{k(k-1)/2}2^k}{k!}
\frac{1}{(2\pi i)^{k}}
\oint \cdots \oint
\frac{G(z_1, \dots,z_{k})\Delta(z_1^2,\dots,z_{k}^2)^2}
{\prod_{j=1}^{k} z_j^{2k-1}}\nonumber\\
& &\label{eq:eq5.1.12} \ \ \ \times q^{\tfrac x2 \sum_{j=1}^{k}z_j}\,dz_1\dots dz_{k} ,
\end{eqnarray}
where $\Delta(z_{1},\ldots,z_{k})$ is defined as in \eqref{eq:vandermonde},
\begin{equation}
G(z_1,\dots,z_k)=A(\tfrac{1}{2};z_1,\dots,z_k) 
\prod_{j=1}^k X(\tfrac12+z_j)^{-\frac12}
%\left(
%    \frac{\Gamma(\frac 34+\frac {z_j}2) 2^{z_j}}
%         {\Gamma(\frac 34- \frac {z_j}2)}
%\right)^{\tfrac 12} 
\prod_{1\le i\le j\le k}\zeta_{A}(1+z_i+z_j),
\end{equation}
and $A(\tfrac{1}{2};z_1,\dots,z_k)$ is the Euler product, absolutely convergent for $|\Re z_j|<\frac12 $, defined by
\begin{eqnarray}
\label{eq:A}
A(\tfrac{1}{2};z_1,\dots,z_k) & = & \prod_{\substack{P \ \mathrm{monic} \\ \mathrm{irreducible}}} \prod_{1\le i \le j \le k}
\left(1-\frac{1}{|P|^{1+z_i+z_j}}\right) \nonumber \\&   &\times \left(\frac
12 \left(\prod_{j=1}^k\left( 1-\frac{1}{|P|^{\frac 12+z_j}}\right)^{-1} +
\prod_{j=1}^k\left(1+\frac{1}{|P|^{\frac12+z_j}}\right)^{-1}
\right)+\frac{1}{|P|} \right)\nonumber\\
& & \times \left( 1+ \frac{1}{|P|}\right)^{-1}. 
\end{eqnarray}
%More generally, we have
%\begin{eqnarray}
%& & \sum_{D\in\mathcal{H}_{2g+1,q}}L(\tfrac{1}{2}+\alpha_{1},\chi_{D})\ldots L(\tfrac{1}{2}+\alpha_{k},\chi_{D})=\nonumber\\
%& = &  \sum_{D\in\mathcal{H}_{2g+1,q}}\prod_{j=1}^{k}X(\tfrac{1}{2}+\alpha_{j})^{-\tfrac{1}{2}}|D|^{-\tfrac{1}{2}\sum_{j=1}^{k}\alpha_{j}}Q_{k}(\log_{q}|D|)(1+O(|D|^{-\frac{1}{2}+\varepsilon})) ,\nonumber\\
%& &\label{eq:eq5.1.15}
%\end{eqnarray}
%in which
%\begin{eqnarray}
%Q_k(x, \alpha)& = & \frac{(-1)^{k(k-1)/2}2^k}{k!} \frac{1}{(2\pi
%i)^{k}} \nonumber \\
%&   &\ \ \times \oint \cdots \oint 
%\frac{G(z_1,
%\dots,z_{k})\Delta(z_1^2,\dots,z_{k}^2)^2 \prod_{j=1}^{k} z_j}
%{\prod_{i =1}^{k} \prod_{j=1}^{k} (z_j -
%\alpha_i)(z_j+\alpha_{i})}\nonumber\\ 
%& & \ \ \times q^{\tfrac x2 \sum_{j=1}^{k}z_j}\,dz_1\dots dz_{k}, 
%\end{eqnarray}
%where the path of integration encloses the $\pm \alpha$'s.
\end{conj}
\begin{rmk}
In the case when $k=1$, this conjecture coincides with a theorem in \cite{AK}. See section 5.1.
\end{rmk}
\begin{rmk}
Note that \eqref{eq:eq5.1.11} is the function field analogue of the formula (1.5.11) in \cite{CFKRS}. 
\end{rmk}

The next conjecture is the translation for function fields of the ratios conjecture for quadratic Dirichlet $L$--functions associated with hyperelliptic curves.

\begin{conj}
\label{conj:ratiosconj}
Suppose that the real parts of $\alpha_{k}$ and $\gamma_{m}$ are positive and that $q$ odd is the fixed cardinality of the finite field $\mathbb{F}_{q}$. Then using the same notations as in the previous conjecture we have
\begin{multline}
\label{eq:6.3.20c}
\sum_{D\in\mathcal{H}_{2g+1,q}}\frac{\prod_{k=1}^{K}L(\tfrac{1}{2}+\alpha_{k},\chi_{D})}{\prod_{m=1}^{Q}L(\tfrac{1}{2}+\gamma_{m},\chi_{D})}\\
=\sum_{D\in\mathcal{H}_{2g+1,q}}\sum_{\epsilon\in\{-1,1\}^{K}}|D|^{\tfrac{1}{2}\sum_{k=1}^{K}(\epsilon_{k}\alpha_{k}-\alpha_{k})}\prod_{k=1}^{K}X\left(\frac{1}{2}+\frac{\alpha_{k}-\epsilon_{k}\alpha_{k}}{2}\right)\\
\times Y(\epsilon_{1}\alpha_{1},\ldots,\epsilon_{K}\alpha_{K};\gamma)A_{\mathcal{D}}(\epsilon_{1}\alpha_{1},\ldots,\epsilon_{K}\alpha_{K},\gamma)+o(|D|),
\end{multline}
where
\begin{multline}
A_{\mathcal{D}}(\alpha;\gamma)=\prod_{\substack{P \ \mathrm{monic} \\ \mathrm{irreducible}}}\frac{\prod_{j\leq k\leq K}\left(1-\frac{1}{|P|^{1+\alpha_{j}+\alpha_{k}}}\right)\prod_{m<r\leq Q}\left(1-\frac{1}{|P|^{1+\gamma_{m}+\gamma_{r}}}\right)}{\prod_{k=1}^{K}\prod_{m=1}^{Q}\left(1-\frac{1}{|P|^{1+\alpha_{k}+\gamma_{m}}}\right)}\\
\times\left(1+\left(1+\frac{1}{|P|}\right)^{-1}\sum_{0<\sum_{k}a_{k}+\sum_{m}c_{m} \ \mathrm{is \ even}}\frac{\prod_{m=1}^{Q}\mu(P^{c_{m}})}{|P|^{\sum_{k}a_{k}\left(\tfrac{1}{2}+\alpha_{k}\right)+\sum_{m}c_{m}\left(\tfrac{1}{2}+\gamma_{m}\right)}}\right)
\end{multline}
and
\begin{equation}
Y(\alpha;\gamma)=\frac{\prod_{j\leq k\leq K}\zeta_{A}(1+\alpha_{j}+\alpha_{k})\prod_{m<r\leq Q}\zeta_{A}(1+\gamma_{m}+\gamma_{r})}{\prod_{k=1}^{K}\prod_{m=1}^{Q}\zeta_{A}(1+\alpha_{k}+\gamma_{m})}.
\end{equation}
If we let,
\begin{equation}
H_{\mathcal{D},|D|,\alpha,\gamma}(w)=|D|^{\tfrac{1}{2}\sum_{k=1}^{K}w_{k}}\prod_{k=1}^{K}X\left(\frac{1}{2}+\frac{\alpha_{k}-w_{k}}{2}\right)Y(w_{1},\ldots,w_{K};\gamma)A_{\mathcal{D}}(w_{1},\ldots,w_{K};\gamma)\\
\end{equation}
then the conjecture may be formulated as
\begin{multline}
\label{eq:6.3.21c}
\sum_{D\in\mathcal{H}_{2g+1,q}}\frac{\prod_{k=1}^{K}L(\tfrac{1}{2}+\alpha_{k},\chi_{D})}{\prod_{m=1}^{Q}L(\tfrac{1}{2}+\gamma_{m},\chi_{D})}\\
=\sum_{D\in\mathcal{H}_{2g+1,q}}|D|^{-\tfrac{1}{2}\sum_{k=1}^{K}\alpha_{k}}\sum_{\epsilon\in\{-1,1\}^{K}}H_{\mathcal{D},|D|,\alpha,\gamma}(\epsilon_{1}\alpha_{1},\ldots,\epsilon_{K}\alpha_{K})+o(|D|).
\end{multline}
\end{conj}

Note that in this paper we are fixing the cardinality $q$ of the ground field $\mathbb{F}_{q}$. The asymptotic formulae we present therefore correspond to letting $g\rightarrow\infty$. This limit is different from that studied by Katz-Sarnak \cite{KS1,KS2}, and coincides with that explored in other contexts by Rudnick and Kurlberg \cite{KR}, Faifman and Rudnick \cite{FR} and Bucur \textit{et al.} in \cite{BucurIMRN}.

\section{Integral Moments of $L$--functions in the Hyperelliptic Ensemble}
\label{sec4}

In this section we will present the details of the recipe for conjecturing moments of $L$--functions associated with hyperelliptic curves of genus $g$ over a fixed finite field $\mathbb{F}_{q}$. To do this we will adapt to the function field setting the recipe presented in \cite{CFKRS}. We note that the recipe is used without rigorous justification in each of its steps, but when seen as a whole it serves to produce a conjecture for the moments of $L$--functions that is consistent with its random matrix analogues and with all results known to date.

Let $D\in\mathcal{H}_{2g+1,q}$. For a fixed $k$, we seek an asymptotic expression for 
\begin{equation}
\label{eq:moment}
\sum_{D\in\mathcal{H}_{2g+1,q}}L(\tfrac{1}{2},\chi_{D})^{k},
\end{equation} 
as $g\rightarrow\infty$. To achieve this we consider the more general expression obtained by introducing small shifts, say $\alpha_{1},\ldots,\alpha_{k}$

\begin{equation}
\sum_{D\in\mathcal{H}_{2g+1,q}}L(\tfrac{1}{2}+\alpha_{1},\chi_{D})\ldots L(\tfrac{1}{2}+\alpha_{k},\chi_{D}).
\end{equation}

By introducing such shifts, hidden structures are revealed in the form of symmetries and the calculations are simplified by the removal of higher order poles. In the end we let each $\alpha_{1},\ldots,\alpha_{k}$ tend to $0$ to recover \eqref{eq:moment}.

\subsection{Some Analogies Between Classical $L$--functions and $L$--functions over Function Fields}

The starting point to conjecture moments for $L$--functions is the use of the approximate functional equation. For the hyperelliptic ensemble considered here, the analogue of the approximate functional equation is given by

\begin{equation}
\label{eq:eq4.1}
L(s,\chi_{D})=\sum_{\substack{n \ \mathrm{monic} \\ \mathrm{deg}(n)\leq g}}\frac{\chi_{D}(n)}{|n|^{s}}+\mathcal{X}_{D}(s)\sum_{\substack{m \ \mathrm{monic} \\ \mathrm{deg}(m)\leq g-1}}\frac{\chi_{D}(m)}{|m|^{1-s}},
\end{equation}
which is an exact formula in this case rather than an approximation, where $D\in\mathcal{H}_{2g+1,q}$ and $\mathcal{X}_{D}(s)=q^{g(1-2s)}$; see \cite{AK} for more details. Note that we can write
\begin{equation}
\mathcal{X}_{D}(s)=|D|^{\tfrac{1}{2}-s}X(s),
\end{equation}
where 

\begin{equation}
X(s)=q^{-\tfrac{1}{2}+s}
\end{equation} 
corresponds to the gamma factor that appears in the classical quadratic $L$--functions. Now we will present some simple lemmas which will be used in the recipe and which make the analogy between the function field case and the number field case more direct.

\begin{lem}
\label{lem:lem4.1}
We have that,

\begin{equation} 
\mathcal{X}_{D}(s)^{1/2}=\mathcal{X}_{D}(1-s)^{-1/2},  
\end{equation}
and
\begin{equation} 
\mathcal{X}_{D}(s)\mathcal{X}_{D}(1-s)=1.  
\end{equation}
\end{lem}

\begin{proof}
The proof is straightforward and follows directly from the definition of $\mathcal{X}_{D}(s)$.
\end{proof}

For ease of presentation, we will work with
\begin{equation}
\label{eq:Z-function}
Z_{L}(s,\chi_{D})=\mathcal{X}_{D}(s)^{-1/2}L(s,\chi_{D}),
\end{equation}
which satisfies a more symmetric functional equation as follows.

\begin{lem}
The function $Z_{L}(s,\chi_{D})$ satisfies the functional equation
\begin{equation}
Z_{L}(s,\chi_{D})=Z_{L}(1-s,\chi_{D}).
\end{equation}
\end{lem}
\begin{proof}
This follows from a direct application of Lemma \ref{lem:lem4.1} part (1).
\end{proof}

We would like to produce an asymptotic for the $k$--shifted moment

\begin{equation}
L_{D}(s)=\sum_{D\in\mathcal{H}_{2g+1,q}}Z(s;\alpha_{1},\ldots,\alpha_{k}),
\end{equation}
where
\begin{equation}
Z(s;\alpha_{1},\ldots,\alpha_{k})=\prod_{j=1}^{k}Z_{L}(s+\alpha_{j},\chi_{D}).
\end{equation}
Making use of \eqref{eq:eq4.1} and Lemma \ref{lem:lem4.1} part(1) we have that 

\begin{equation}
\label{eq:4.12}
Z_{L}(s,\chi_{D})=\mathcal{X}_{D}(s)^{-1/2}\sum_{\substack{n \ \mathrm{monic} \\ \mathrm{deg}(n)\leq g}}\frac{\chi_{D}(n)}{|n|^{s}}+\mathcal{X}_{D}(1-s)^{-1/2}\sum_{\substack{m \ \mathrm{monic} \\ \mathrm{deg}(m)\leq g-1}}\frac{\chi_{D}(m)}{|m|^{1-s}}.
\end{equation}

\subsection{Adapting the CFKRS Recipe for the function field case}
We will follow \cite[section 4]{CFKRS} making adjustments for function fields when necessary.

(1) We start with a product of $k$ shifted $L$--functions:

\begin{equation}
Z(s;\alpha_{1}\ldots,\alpha_{k})=Z_{L}(s+\alpha_{1},\chi_{D})\ldots Z_{L}(s+\alpha_{k},\chi_{D}).
\end{equation}

(2) Replace each $L$--function by its corresponding ``approximate" functional equation \eqref{eq:4.12}. Hence we obtain, %Multiply out to get an expression of the form: 

\begin{equation}
Z(\tfrac{1}{2};\alpha_{1},\ldots,\alpha_{k})=\sum_{\varepsilon_{j}=\pm1}\prod_{j=1}^{k}\left(\mathcal{X}_{D}(\tfrac{1}{2}+\varepsilon_{j}\alpha_{j})^{-1/2}\sum_{\substack{n_{j} \ \mathrm{monic} \\ \mathrm{deg}(n_{j})\leq f(\varepsilon_{j})}}\frac{\chi_{D}(n_{j})}{|n_{j}|^{\tfrac{1}{2}+\varepsilon_{j}\alpha_{j}}}\right),
\end{equation}
where $f(1)=g$ and $f(-1)=g-1$. We then multiply out and end up with,

\begin{equation}
\label{eq:eq4.16}
Z(\tfrac{1}{2};\alpha_{1},\ldots,\alpha_{k})=\sum_{\varepsilon_{j}=\pm1}\prod_{j=1}^{k}\mathcal{X}_{D}(\tfrac{1}{2}+\varepsilon_{j}\alpha_{j})^{-1/2}\sum_{\substack{n_{1},\ldots,n_{k} \\ n_{i} \ \mathrm{monic} \\ \mathrm{deg}(n_{j})\leq f(\varepsilon_{j})}}\frac{\chi_{D}(n_{1}\ldots n_{k})}{\prod_{j=1}^{k}|n_{j}|^{\tfrac{1}{2}+\varepsilon_{j}\alpha_{j}}}.
\end{equation}

(3) Average the sign of the functional equations. 
%Replace the product of $\varepsilon_{f}$--factors by its average over the family.

Note that in this case the signs of the functional equations are all equal to $1$ and therefore do not produce any effect on the final result.

(4) Replace each summand by its expected value when averaged over the family $\mathcal{H}_{2g+1,q}$. 

In this step we need to average over all fundamental discriminants $D\in\mathcal{H}_{2g+1,q}$ and as a preliminary task, we will restate the following orthogonality relation for quadratic Dirichlet characters over function fields.

\begin{lem}
\label{lem:lem4.3}
Let
\begin{equation}
a_{m}=\prod_{\substack{P \ \mathrm{monic} \\ \mathrm{irreducible} \\ P\mid m}}\left(1+\frac{1}{|P|}\right)^{-1}.
\end{equation}
Then,
\begin{equation}
\label{eq:constraint}
\lim_{\mathrm{deg}(D)\rightarrow\infty}\frac{1}{\#\mathcal{H}_{2g+1,q}}\sum_{D\in\mathcal{H}_{2g+1,q}}\chi_{D}(m)=\begin{cases}
a_{m} \ \ \mathrm{if} & m \ \mathrm{is \ the \ square \ of \ a \ polynomial} \\
0& \mathrm{otherwise}. 
\end{cases} 
\end{equation}
(for short hand we will use the notation $m=\square$ when $m$ is the square of a polynomial).
\end{lem}
\begin{proof}

We start by considering $m=\square=l^{2}$, then using Proposition 2 from \cite{AK} and the fact that $\tfrac{\Phi(l)}{|l|}\leq1$ we have, 
\begin{equation}
\frac{1}{\#\mathcal{H}_{2g+1,q}}\sum_{D\in\mathcal{H}_{2g+1,q}}\chi_{D}(m=l^{2})=\frac{1}{\#\mathcal{H}_{2g+1,q}}\frac{|D|}{\zeta_{A}(2)}\prod_{\substack{P \ \mathrm{monic} \\ \mathrm{irreducible} \\ P\mid m}}\left(1+\frac{1}{|P|}\right)^{-1}+O\left(\frac{\sqrt{|D|}}{\#\mathcal{H}_{2g+1,q}}\right).
\end{equation}
By making use of equation \eqref{eq:3.17} we obtain

\begin{equation}
\frac{1}{\#\mathcal{H}_{2g+1,q}}\sum_{D\in\mathcal{H}_{2g+1,q}}\chi_{D}(m=l^{2})=\prod_{\substack{P \ \mathrm{monic} \\ \mathrm{irreducible} \\ P\mid m}}\left(1+\frac{1}{|P|}\right)^{-1} +O(q^{-g}).
\end{equation}
%$$\sum_{D\in\mathcal{H}_{2g+1,q}}\chi_{D}(m)=\frac{|D|}{\zeta_{A}(2)}\prod_{\substack{P \ \mathrm{monic} \\ \mathrm{irreducible} \\ P\mid m}}\left(1+\frac{1}{|P|}\right)^{-1}+O(\sqrt{|D|}).$$
%So,
%\begin{eqnarray}
%\frac{1}{\#\mathcal{H}_{2g+1,q}}\sum_{D\in\mathcal{H}_{2g+1,q}}\chi_{D}(m=l^{2}) & = & \frac{|D|}{\zeta_{A}(2)}\frac{1}{(q-1)q^{2g}}\prod_{\substack{P \ \mathrm{monic} \\ \mathrm{irreducible} \\ P\mid m}}\left(1+\frac{1}{|P|}\right)^{-1}\nonumber\\
%& + & O(\sqrt{|D|}((q-1)q^{2g})^{-1})\nonumber\\
%& = & \prod_{\substack{P \ \mathrm{monic} \\ \mathrm{irreducible} \\ P\mid m}}\left(1+\frac{1}{|P|}\right)^{-1} +O(q^{-g}).\nonumber  
%\end{eqnarray}
Therefore,
\begin{equation}
\lim_{\deg(D)\rightarrow\infty}\frac{1}{\#\mathcal{H}_{2g+1,q}}\sum_{D\in\mathcal{H}_{2g+1,q}}\chi_{D}(m=l^{2})=\prod_{\substack{P \ \mathrm{monic} \\ \mathrm{irreducible} \\ P\mid m}}\left(1+\frac{1}{|P|}\right)^{-1}.
\end{equation}

If $m\neq\square$ we can use the function field version of the Polya--Vinogradov inequality \cite[Lemma 2.1]{FR} to bound the sum over non--trivial Dirichlet characters,  
\begin{equation}
\Bigg|\sum_{\substack{D\in\mathcal{H}_{2g+1,q} \\ m\neq\square}}\chi_{D}(m)\Bigg|\ll 2^{\mathrm{deg}(m)}\sqrt{|D|},
\end{equation}
and so we end up with,
\begin{eqnarray}
\frac{1}{\#\mathcal{H}_{2g+1,q}}\sum_{\substack{D\in\mathcal{H}_{2g+1,q} \\ m\neq\square}}\chi_{D}(m) & \ll & \frac{2^{g}\sqrt{|D|}}{(q-1)q^{2g}}\nonumber\\
& \ll & q^{-g}2^{g},
\end{eqnarray}
which tends to zero when $g\rightarrow\infty$ since $q>1$ is a fixed odd number. %Note that here we have used the hypothesis that $q>3$ because for $q=3$ we have 
%\begin{equation}
%\frac{1}{\#\mathcal{H}_{2g+1,3}}\sum_{\substack{D\in\mathcal{H}_{2g+1,3} \\ m\neq\square}}\chi_{D}(m)\ll 3^{-g}2^{g},
%\end{equation}
%which tends to infinity as $g\rightarrow\infty$.

%the error term above gets bigger than the main term. 
%and the asymptotic formula it is not true in this case. So the constraint $q>3$ warrant us a good error term in the asymptotic formula given by \eqref{eq:constraint}.
\end{proof}

Using Lemma \ref{lem:lem4.3}, we can average the summand in \eqref{eq:eq4.16}, since
\begin{equation}
\lim_{g\rightarrow\infty}\langle  \chi_D(n_{1})\ldots\chi_{D}(n_{k}) \rangle
=
\begin{cases}
\prod_{P|\square} \left(1+\frac{1}{|P|}\right)^{-1} & \ \ \mathrm{if} \ \ n_{1}\ldots n_{k} =\square, \\
0& \text{otherwise}. 
\end{cases} 
\end{equation}
We therefore write (heuristically as the sums below can diverge depending on the choice of $\varepsilon_{j}\alpha_{j}$'s)

\begin{eqnarray}
\lim_{g\rightarrow\infty}\frac{1}{\#\mathcal{H}_{2g+1,q}}\sum_{D\in\mathcal{H}_{2g+1,q}}\sum_{\substack{n_{1},\ldots,n_{k} \\ n_{i} \ \mathrm{monic}}}\frac{\chi_{D}(n_{1}\ldots n_{k})}{\prod_{j=1}^{k}|n_{j}|^{\tfrac{1}{2}+\varepsilon_{j}\alpha_{j}}} & = & \sum_{\substack{n_{1},\ldots,n_{k} \\ n_{i} \ \mathrm{monic} \\ n_{1}\ldots n_{k}=m^{2}}}\frac{a_{m^{2}}}{\prod_{j=1}^{k}|n_{j}|^{\tfrac{1}{2}+\varepsilon_{j}\alpha_{j}}}\nonumber\\
& = & \sum_{m \ \mathrm{monic}}\sum_{\substack{n_{1},\ldots,n_{k} \\ n_{i} \ \mathrm{monic} \\ n_{1}\ldots n_{k}=m^{2}}}\frac{a_{m^{2}}}{\prod_{j=1}^{k}|n_{j}|^{\tfrac{1}{2}+\varepsilon_{j}\alpha_{j}}}.
\end{eqnarray}

(5) Extend, in \eqref{eq:eq4.16}, each of $n_{1},\ldots,n_{k}$ for all monic polynomials and denote the result $M(s;\alpha_{1},\ldots,\alpha_{k})$ to produce the desired conjecture.

If we call

\begin{equation}
R_{k}\left(\frac{1}{2};\varepsilon_{1}\alpha_{1},\ldots,\varepsilon_{k}\alpha_{k}\right)=\sum_{m \ \mathrm{monic}}\sum_{\substack{n_{1},\ldots,n_{k} \\ n_{i} \ \mathrm{monic} \\ n_{1}\ldots n_{k}=m^{2}}}\frac{a_{m^{2}}}{\prod_{j=1}^{k}|n_{j}|^{\tfrac{1}{2}+\varepsilon_{j}\alpha_{j}}},
\end{equation}
the recipe thus predicts

\begin{equation}
\label{thm:genconj1}
\sum_{D\in\mathcal{H}_{2g+1,q}}Z(\tfrac{1}{2},\alpha_{1},\ldots,\alpha_{k})=\sum_{D\in\mathcal{H}_{2g+1,q}}M(\tfrac{1}{2},\alpha_{1},\ldots,\alpha_{k})(1+o(1)),
\end{equation}
where

\begin{equation}
M\left(\frac{1}{2};\alpha_{1}\ldots\alpha_{k}\right)=\sum_{\varepsilon_{j}=\pm1}\prod_{j=1}^{k}\mathcal{X}_{D}(\tfrac{1}{2}+\varepsilon_{j}\alpha_{j})^{-1/2}R_{k}\left(\frac{1}{2};\varepsilon_{1}\alpha_{1},\ldots,\varepsilon_{k}\alpha_{k}\right).
\end{equation}

%we have that the quantity produced by the recipe is

%\begin{equation}
%M\left(\frac{1}{2};\alpha_{1}\ldots\alpha_{k}\right)=\sum_{\varepsilon_{j}=\pm1}\prod_{j=1}^{k}\mathcal{X}_{D}(\tfrac{1}{2}+\varepsilon_{j}\alpha_{j})^{-1/2}R_{k}\left(\frac{1}{2};\varepsilon_{1}\alpha_{1},\ldots,\varepsilon_{k}\alpha_{k}\right).
%\end{equation}

%(6) The conjecture is

%\begin{equation}
%\label{thm:genconj1}
%\sum_{D\in\mathcal{H}_{2g+1,q}}Z(\tfrac{1}{2},\alpha_{1},\ldots,\alpha_{k})=\sum_{D\in\mathcal{H}_{2g+1,q}}M(\tfrac{1}{2},\alpha_{1},\ldots,\alpha_{k})(1+o(1)).
%\end{equation}

\subsection{Putting the Conjecture in a More Useful Form}

The conjecture \eqref{thm:genconj1} is problematic in the form presented because the individual terms have poles that cancel when summed. In this section we put it in a more useful form, writing $R_{k}$ as an Euler product and then factoring out the appropriate $\zeta_{A}(s)$--factors.

We have that $a_{m}$ is multiplicative, since
\begin{equation}
a_{mn}=a_{m}a_{n} \ \ \ \ \text{whenever $\mathrm{gcd}(m,n)=1$},
\end{equation}
where
\begin{equation}
a_{m}=\prod_{\substack{P \ \mathrm{monic} \\ \mathrm{irreducible} \\ P\mid m}}(1+|P|^{-1})^{-1},
\end{equation}
and if we define
\begin{equation}
\psi(x):=\sum_{\substack{n_{1}\ldots n_{k}=x \\ n_{i} \ \mathrm{monic}}}\frac{1}{|n_{1}|^{s+\alpha_{1}}\ldots|n_{k}|^{s+\alpha_{k}}},
\end{equation}
we have that $\psi(m^{2})$ is multiplicative on $m$.

So,
\begin{eqnarray}
& & \sum_{m \ \mathrm{monic}}\sum_{\substack{n_{1},\ldots,n_{k} \\ n_{i} \ \mathrm{monic} \\ n_{1}\ldots n_{k}=m^{2}}}\frac{a_{m^{2}}}{|n_{1}|^{s+\alpha_{1}}\ldots|n_{k}|^{s+\alpha_{k}}}=\sum_{m \ \mathrm{monic}}a_{m^{2}}\sum_{\substack{n_{1},\ldots,n_{k} \\ n_{i} \ \mathrm{monic} \\ n_{1}\ldots n_{k}=m^{2}}}\frac{1}{|n_{1}|^{s+\alpha_{1}}\ldots|n_{k}|^{s+\alpha_{k}}}\\
& = & \sum_{m \ \mathrm{monic}}a_{m^{2}}\psi(m^{2})=\prod_{\substack{P \ \mathrm{monic} \\ \mathrm{irreducible}}}\left(1+\sum_{j=1}^{\infty}a_{P^{2j}}\psi(P^{2j})\right),\nonumber
\end{eqnarray}
where
\begin{equation}
\psi(P^{2j})=\sum_{\substack{n_{1},\ldots,n_{k} \\ n_{i} \ \mathrm{monic} \\ n_{1}\ldots n_{k}=P^{2j}}}\frac{1}{|n_{1}|^{s+\alpha_{1}}\ldots|n_{k}|^{s+\alpha_{k}}},
\end{equation}
and so, $n_{i}=P^{e_{i}}$, for $i=1,\ldots,k$ and $e_{1}+\cdots+e_{k}=2j$. %due the fact that $|P|^{e_{1}}\ldots|P|^{e_{k}}=|P|^{e_{1}+\cdots+e_{k}}=|P|^{2j}$. 

Hence we can write

\begin{equation}
\psi(P^{2j})=\sum_{\substack{e_{1},\ldots,e_{k}\geq0 \\ e_{1}+\cdots+e_{k}=2j}}\prod_{i=1}^{k}\frac{1}{|P|^{e_{i}(s+\alpha_{i})}}
\end{equation}
%\begin{eqnarray}
%\psi(P^{2j}) & = & \sum_{\substack{e_{1},\ldots,e_{k}\geq0 \\ e_{1}+\cdots+e_{k}=2j}}\frac{1}{|P|^{e_{1}(s+\alpha_{1})}\ldots|P|^{e_{k}(s+\alpha_{k})}}\nonumber\\
%& = & \sum_{\substack{e_{1},\ldots,e_{k}\geq0 \\ e_{1}+\cdots+e_{k}=2j}}\prod_{i=1}^{k}\frac{1}{|P|^{e_{i}(s+\alpha_{i})}}\nonumber
%\end{eqnarray}
and thus we end up with 
\begin{multline}
\label{eq:eq5.3.1}
R_{k}(s;\alpha_{1},\ldots,\alpha_{k})=\prod_{\substack{P \ \mathrm{monic} \\ \mathrm{irreducible}}}\left(1+\sum_{j=1}^{\infty}a_{P^{2j}}\psi(P^{2j})\right)\\
=\prod_{\substack{P \ \mathrm{monic} \\ \mathrm{irreducible}}}\left(1+\sum_{j=1}^{\infty}a_{P^{2j}}\sum_{\substack{e_{1},\ldots,e_{k}\geq0 \\ e_{1}+\cdots+e_{k}=2j}}\prod_{i=1}^{k}\frac{1}{|P|^{e_{i}(s+\alpha_{i})}}\right).\ \ \ \ \ \ \ \ \ \ \ \ \ \ \ \ \ 
\end{multline}
But
\begin{equation}
a_{P^{2j}}=(1+|P|^{-1})^{-1},
\end{equation}
so that \eqref{eq:eq5.3.1} becomes
\begin{eqnarray}
R_{k}(s;\alpha_{1},\ldots,\alpha_{k})&=& \prod_{\substack{P \ \mathrm{monic} \\ \mathrm{irreducible}}}\left(1+(1+|P|^{-1})^{-1}\sum_{j=1}^{\infty}\sum_{\substack{e_{1},\ldots,e_{k}\geq0 \\ e_{1}+\cdots+e_{k}=2j}}\prod_{i=1}^{k}\frac{1}{|P|^{e_{i}(s+\alpha_{i})}}\right)\nonumber\\
& = & \prod_{\substack{P \ \mathrm{monic} \\ \mathrm{irreducible}}}R_{k,P}.
\end{eqnarray}

Using 

\begin{equation}
(1+|P|^{-1})^{-1}=\sum_{l=0}^{\infty}\frac{(-1)^{l}}{|P|^{l}}
\end{equation}
%$$(1+|P|^{-1})^{-1}=1-\frac{1}{|P|}+\frac{1}{|P|^{2}}-\frac{1}{|P|^{3}}+\cdots=\sum_{l=0}^{\infty}\frac{(-1)^{l}}{|P|^{l}}$$
we have that
\begin{equation}
R_{k,P}=1+\sum_{l=0}^{\infty}\sum_{j=1}^{\infty}\sum_{\substack{e_{1},\ldots,e_{k}\geq0 \\ e_{1}+\cdots+e_{k}=2j}}\prod_{i=1}^{k}\frac{(-1)^{l}}{|P|^{e_{i}(s+\alpha_{i})+l}}
\end{equation}
and so
\begin{equation}
R_{k}(s;\alpha_{1},\ldots,\alpha_{k})=\prod_{\substack{P \ \mathrm{monic} \\ \mathrm{irreducible}}}\left(1+\sum_{l=0}^{\infty}\sum_{j=1}^{\infty}\sum_{\substack{e_{1},\ldots,e_{k}\geq0 \\ e_{1}+\cdots+e_{k}=2j}}\prod_{i=1}^{k}\frac{(-1)^{l}}{|P|^{e_{i}(s+\alpha_{i})+l}}\right).
\end{equation}

The key point is that when $\alpha_{i}=0$ and $s=1/2$ only terms with $e_{1}+\cdots+e_{k}=2$ give rise to poles. Isolating the term with $l=0$ and $j=1$: %Thus, we look for $l=0$ and $j=1$ and have 

\begin{eqnarray}
R_{k,P} & = & 1+\sum_{e_{1}+\cdots+e_{k}=2}\prod_{i=1}^{k}\frac{1}{|P|^{e_{i}(s+\alpha_{i})}}+ \ (\mathrm{lower \ order \ terms})\nonumber\\
%& = & 1+\sum_{e_{1}+\cdots+e_{k}=2}\left(\frac{1}{|P|^{e_{1}(s+\alpha_{1})}}\cdots\frac{1}{|P|^{e_{k}(s+\alpha_{k})}}\right)+ \ (\mathrm{lower \ order \ terms})\nonumber\\
%& = & 1+\frac{1}{|P|^{(s+\alpha_{1})}|P|^{s+\alpha_{2}}}+\frac{1}{|P|^{s+\alpha_{1}}|P|^{s+\alpha_{3}}}+\cdots+ \ (\mathrm{lower \ order \ terms})\nonumber\\
& = & 1+\sum_{1\leq i\leq j\leq k}\frac{1}{|P|^{2s+\alpha_{i}+\alpha_{j}}}+ \ (\mathrm{lower \ order \ terms})
\end{eqnarray}
Hence we can write, for $\mathfrak{R}(\alpha_{i})$ sufficiently small,
\begin{equation}
R_{k,P}=1+\sum_{1\leq i\leq j\leq k}\frac{1}{|P|^{2s+\alpha_{i}+\alpha_{j}}}+O(|P|^{-1-2s+\varepsilon})+O(|P|^{-3s+\varepsilon})
\end{equation}
(for more details see \cite[pg87]{CFKRS}). Expressing $R_{k,P}$ as a product, we finish with
\begin{equation}
R_{k,P}=\prod_{1\leq i\leq j\leq k}\left(1+\frac{1}{|P|^{2s+\alpha_{i}+\alpha_{j}}}\right)\times(1+O(|P|^{-1-2s+\varepsilon})+O(|P|^{-3s+\varepsilon})).
\end{equation}

Now, since
\begin{equation}\prod_{\substack{P \ \mathrm{monic} \\ \mathrm{irreducible}}}\left(1+\frac{1}{|P|^{2s}}\right)=\frac{\zeta_{A}(2s)}{\zeta_{A}(4s)}
\end{equation}
has a simple pole at $s=\frac{1}{2}$ and
\begin{equation}
\prod_{\substack{P \ \mathrm{monic} \\ \mathrm{irreducible}}}(1+O(|P|^{-1-2s+\varepsilon})+O(|P|^{-3s+\varepsilon}))
\end{equation}
is analytic in $\mathfrak{R}(s)>\frac{1}{3}$, we see that $\prod_{P}R_{k,P}$ has a pole at $s=\frac{1}{2}$ of order $k(k+1)/2$ if $\alpha_{1}=\cdots=\alpha_{k}=0$.

With the divergent sums replaced by their analytic continuation and the leading order poles clearly identified, we are almost ready to put conjecture \ref{thm:genconj1} in a more desirable form. We just need to factor out the appropriate zeta--factors and write the above product $\prod_{P}R_{k,P}$ as

\begin{equation}
R_{k}(s;\alpha_{1},\ldots,\alpha_{k})=\prod_{1\leq i\leq j\leq k}\zeta_{A}(2s+\alpha_{i}+\alpha_{j})A(s;\alpha_{1},\ldots,\alpha_{k}),
\end{equation}
where
\begin{equation}
A(s;\alpha_{1},\ldots,\alpha_{k})=\prod_{\substack{P \ \mathrm{monic} \\ \mathrm{irreducible}}}\left(R_{k,P}(s;\alpha_{1},\ldots,\alpha_{k})\prod_{1\leq i\leq j\leq k}\left(1-\frac{1}{|P|^{2s+\alpha_{i}+\alpha_{j}}}\right)\right).
\end{equation}
%\begin{eqnarray}
%& & R_{k}(s;\alpha_{1},\ldots,\alpha_{k})\nonumber\\
%& = & \prod_{\substack{P \ \mathrm{monic} \\ \mathrm{irreducible}}}\left(\prod_{1\leq i\leq j\leq k}\left(1+\frac{1}{|P|^{2s+\alpha_{i}+\alpha_{j}}}\right)\times(1+O(|P|^{-1-2s+\varepsilon})+O(|P|^{-3s+\varepsilon}))\right)\nonumber\\
%& = & \prod_{1\leq i\leq j\leq k}\frac{\zeta_{A}(2s+\alpha_{i}+\alpha_{j})}{\zeta_{A}(2(2s+\alpha_{i}+\alpha_{j}))}\prod_{\substack{P \ \mathrm{monic} \\ \mathrm{irreducible}}}(1+O(|P|^{-1-2s+\varepsilon})+O(|P|^{-3s+\varepsilon}))\nonumber\\
%& = & \prod_{1\leq i\leq j\leq k}\zeta_{A}(2s+\alpha_{i}+\alpha_{j})\prod_{\substack{P \ \mathrm{monic} \\ \mathrm{irreducible}}}\prod_{1\leq i\leq j\leq k}\left(1-\frac{1}{|P|^{2s+\alpha_{i}+\alpha_{j}}}\right)\left(1+\frac{1}{|P|^{2s+\alpha_{i}+\alpha_{j}}}\right)\nonumber\\
%& & \times(1+O(|P|^{-1-2s+\varepsilon})+O(|P|^{-3s+\varepsilon}))\nonumber\\
%& = & \prod_{1\leq i\leq j\leq k}\zeta_{A}(2s+\alpha_{i}+\alpha_{j})\prod_{\substack{P \ \mathrm{monic} \\ \mathrm{irreducible}}}\left(R_{k,P}(s;\alpha_{1},\ldots,\alpha_{k})\prod_{1\leq i\leq j\leq k}\left(1-\frac{1}{|P|^{2s+\alpha_{i}+\alpha_{j}}}\right)\right)\nonumber\\
%& = &\prod_{1\leq i\leq j\leq k}\zeta_{A}(2s+\alpha_{i}+\alpha_{j})A(s;\alpha_{1},\ldots,\alpha_{k}).\nonumber
%\end{eqnarray}
Here, $A(s;\alpha_{1},\ldots,\alpha_{k})$ defines an absolutely convergent Dirichlet series for $\mathfrak{R}(s)=\frac{1}{2}$ and for all $\alpha_{j}^{'}$s positive. Consequently, we have

\begin{multline}
M\left(\frac{1}{2};\alpha_{1},\ldots,\alpha_{k}\right)\\
=\sum_{\varepsilon_{j}=\pm1}\prod_{j=1}^{k}\mathcal{X}_{D}(\tfrac{1}{2}+\varepsilon_{j}\alpha_{j})^{-\tfrac{1}{2}}\prod_{1\leq i\leq j\leq k}\zeta_{A}(1+\varepsilon_{i}\alpha_{i}+\varepsilon_{j}\alpha_{j})A(\tfrac{1}{2};\varepsilon_{1}\alpha_{1},\ldots,\varepsilon_{k}\alpha_{k}),
\end{multline}
and so the conjectured asymptotic takes the form

\begin{multline}
\label{eq:eq5.3.4}
\sum_{D\in\mathcal{H}_{2g+1,q}}Z(\tfrac{1}{2},\alpha_{1},\ldots,\alpha_{k})\\
=\sum_{D\in\mathcal{H}_{2g+1,q}}\sum_{\varepsilon_{j}=\pm1}\prod_{j=1}^{k}\mathcal{X}_{D}(\tfrac{1}{2}+\varepsilon_{j}\alpha_{j})^{-\tfrac{1}{2}}A(\tfrac{1}{2};\varepsilon_{1}\alpha_{1},\ldots,\varepsilon_{k}\alpha_{k})\prod_{1\leq i\leq j\leq k}\zeta_{A}(1+\varepsilon_{i}\alpha_{1}+\varepsilon_{j}\alpha_{j})(1+o(1)).
\end{multline}

Using the definition of $\mathcal{X}_{D}(s)$, we have that

\begin{equation}
\mathcal{X}_{D}(\tfrac{1}{2}+\varepsilon_{j}\alpha_{j})^{-\tfrac{1}{2}}=|D|^{\tfrac{\varepsilon_{j}\alpha_{j}}{2}}X(\tfrac{1}{2}+\varepsilon_{j}\alpha_{j})^{-\tfrac{1}{2}},
\end{equation}
and substituting this into \eqref{eq:eq5.3.4}, after some arithmetical manipulations we are led to the following form of the conjecture:

\begin{multline}
\label{eq:eqfinalconj}
\sum_{D\in\mathcal{H}_{2g+1,q}}Z_{L}(\tfrac{1}{2}+\alpha_{1},\chi_{D})\ldots Z_{L}(\tfrac{1}{2}+\alpha_{k},\chi_{D})\\
=\sum_{\varepsilon_{j}\pm1}\prod_{j=1}^{k}X(\tfrac{1}{2}+\varepsilon_{j}\alpha_{j})^{-1/2}\sum_{D\in\mathcal{H}_{2g+1,q}}R_{k}(\tfrac{1}{2};\varepsilon_{1}\alpha_{1},\ldots,\varepsilon_{k}\alpha_{k})|D|^{\left(\tfrac{1}{2}\right)\sum_{j=1}^{k}\varepsilon_{j}\alpha_{j}}(1+o(1)).
\end{multline}

Note that \eqref{eq:eqfinalconj} is the function field analogue of the formula (4.4.22) in \cite{CFKRS}. 
% and it presents the analogue of the Euler product $R_{k,N}(\tfrac{1}{2},\varepsilon_{1}\alpha_{1},\ldots,\varepsilon_{k}\alpha_{k})$ for polynomials.

\subsection{The Contour Integral Representation of the Conjecture}

In this section we will use the following lemma from \cite{CFKRS}.

\begin{lem}[Conrey, Farmer, Keating, Rubinstein, Snaith] \label{thm:concisesumsymplectic}

Suppose $F$ is a symmetric function of $k$
variables, regular near $(0,\ldots,0)$, and that $f(s)$ has a simple
pole of residue~$1$ at $s=0$ and is otherwise analytic in a
neighborhood of $s=0$, and let
\begin{equation}
K(a_1,\ldots,a_k)=F(a_1,\ldots,a_k) \prod_{1\leq i \leq j\leq k}
f(a_i+a_j)
\end{equation}
or
\begin{equation}
K(a_1,\ldots,a_k)=F(a_1,\ldots,a_k) \prod_{1\leq i < j\leq k}
f(a_i+a_j).
\end{equation}
If $\alpha_i+\alpha_j$ are contained in the
region of analyticity of $f(s)$ then
\begin{eqnarray}
\sum_{\epsilon_j =\pm 1 }K(\epsilon_1 \alpha_1,\ldots,
\epsilon_k\alpha_k) & = &\frac{(-1)^{k(k-1)/2} } {(2\pi
i)^k} \frac{2^k}{k!} \oint \cdots \oint K(z_1,\ldots,z_k)\nonumber\\
& \times & \frac{\Delta(z_1^2,\ldots,z_k^2)^2 \prod_{j=1}^k z_j }
{\prod_{i=1}^k\prod_{j=1}^k (z_i-\alpha_j)(z_i+\alpha_j)}
\,dz_1\cdots dz_k , 
\end{eqnarray} and
\begin{eqnarray}
& & \sum_{\epsilon_j=\pm 1  }\left(\prod_{j=1}^k
\epsilon_j\right)K(\epsilon_1 \alpha_1,\ldots, \epsilon_k\alpha_k)\nonumber \\
& &\ \ \ \ =\frac{(-1)^{k(k-1)/2} } {(2\pi i)^k} \frac{2^k}{k!} \oint
\cdots \oint K(z_1,\ldots,z_k)\nonumber\\
& & \ \ \ \ \ \ \times\frac{\Delta(z_1^2,\ldots,z_k^2)^2
\prod_{j=1}^k \alpha_j } {\prod_{i=1}^k\prod_{j=1}^k
(z_i-\alpha_j)(z_i+\alpha_j)} \,dz_1\cdots dz_k, 
\end{eqnarray}
where the path of integration encloses the $\pm \alpha_j$'s.
\end{lem}

We will use this lemma to write conjecture \eqref{eq:eqfinalconj} for function fields as a contour integral. For this, note that

\begin{multline}
\sum_{D\in\mathcal{H}_{2g+1,q}}Z_{L}(\tfrac{1}{2}+\alpha_{1},\chi_{D})\ldots Z_{L}(\tfrac{1}{2}+\alpha_{k},\chi_{D})\\
=\sum_{D\in\mathcal{H}_{2g+1,q}}\prod_{j=1}^{k}\mathcal{X}_{D}(\tfrac{1}{2}+\alpha_{j})^{-1/2}L(\tfrac{1}{2}+\alpha_{1},\chi_{D})\ldots L(\tfrac{1}{2}+\alpha_{k},\chi_{D})
\end{multline}
and as $\mathcal{X}_{D}(\tfrac{1}{2}+\alpha_{j})^{-1/2}$ depends only on $|D|$, which is the same for all $D\in\mathcal{H}_{2g+1,q}$, we can factor it out, so that \eqref{eq:eqfinalconj} becomes

\begin{eqnarray}
& & \sum_{D\in\mathcal{H}_{2g+1,q}}L(\tfrac{1}{2}+\alpha_{1},\chi_{D})\ldots L(\tfrac{1}{2}+\alpha_{k},\chi_{D})\nonumber\\
%& = & \sum_{D\in\mathcal{H}_{2g+1,q}}\sum_{\varepsilon_{j}=\pm1}\prod_{j=1}^{k}X(\tfrac{1}{2}+\varepsilon_{j}\alpha_{j})^{-1/2}\prod_{j=1}^{k}\mathcal{X}_{D}(\tfrac{1}{2}+\varepsilon_{j}\alpha_{j})^{1/2}\nonumber\\
%& & \ \ \ \ \times R_{k}(\tfrac{1}{2};\varepsilon_{1}\alpha_{1},\ldots,\varepsilon_{k}\alpha_{k})|D|^{\tfrac{1}{2}\sum_{j=1}^{k}\varepsilon_{j}\alpha_{j}}(1+o(1))\nonumber\\
%& = & \sum_{D\in\mathcal{H}_{2g+1,q}}\prod_{j=1}^{k}X(\tfrac{1}{2}+\alpha_{j})|D|^{-\tfrac{1}{2}\sum_{j=1}^{k}\alpha_{j}}\sum_{\varepsilon_{j}=\pm1}\prod_{j=1}^{k}X(\tfrac{1}{2}+\varepsilon_{j}\alpha_{j})^{-1/2}R_{k}(\tfrac{1}{2};\varepsilon_{1}\alpha_{1},\ldots,\varepsilon_{k}\alpha_{k})\nonumber\\
%& & \ \ \ \ \times|D|^{\tfrac{1}{2}\sum_{j=1}^{k}\varepsilon_{j}\alpha_{j}}(1+O(|D|^{-1/2+\varepsilon}))\nonumber\\
& = & \sum_{D\in\mathcal{H}_{2g+1,q}}\prod_{j=1}^{k}X(\tfrac{1}{2}+\alpha_{j})|D|^{-\tfrac{1}{2}\sum_{j=1}^{k}\alpha_{j}}\sum_{\varepsilon_{j}=\pm1}\prod_{j=1}^{k}X(\tfrac{1}{2}+\varepsilon_{j}\alpha_{j})^{-1/2}\nonumber\\
& & \ \ \ \ \times A\left(\tfrac{1}{2};\varepsilon_{1}\alpha_{1},\ldots,\varepsilon_{k}\alpha_{k}\right)|D|^{\tfrac{1}{2}\sum_{j=1}^{k}\varepsilon_{j}\alpha_{j}}\prod_{1\leq i\leq j\leq k}\zeta_{A}(1+\varepsilon_{i}\alpha_{i}+\varepsilon_{j}\alpha_{j})(1+o(1)).
\end{eqnarray}
Hence, taking out a factor of $\log q$ from each term in the second product

\begin{eqnarray}
& & \sum_{D\in\mathcal{H}_{2g+1,q}}L(\tfrac{1}{2}+\alpha_{1},\chi_{D})\ldots L(\tfrac{1}{2}+\alpha_{k},\chi_{D})\nonumber\\
& = & \sum_{D\in\mathcal{H}_{2g+1,q}}\frac{\prod_{j=1}^{k}X(\tfrac{1}{2}+\alpha_{j})|D|^{-\tfrac{1}{2}\sum_{j=1}^{k}\alpha_{j}}}{(\log q)^{k(k+1)/2}}\sum_{\varepsilon_{j}=\pm1}\prod_{j=1}^{k}X(\tfrac{1}{2}+\varepsilon_{j}\alpha_{j})^{-1/2}A\left(\tfrac{1}{2};\varepsilon_{1}\alpha_{1},\ldots,\varepsilon_{k}\alpha_{k}\right)\nonumber\\
& & \label{eq:eqteste} \ \ \ \ \times|D|^{\tfrac{1}{2}\sum_{j=1}^{k}\varepsilon_{j}\alpha_{j}}\prod_{1\leq i\leq j\leq k}\zeta_{A}(1+\varepsilon_{i}\alpha_{i}+\varepsilon_{j}\alpha_{j})(\log q)(1+o(1)).
\end{eqnarray}
If we call 
\begin{equation}
F(\alpha_{1},\ldots,\alpha_{k})=\prod_{j=1}^{k}X(\tfrac{1}{2}+\alpha_{j})^{-1/2}A\left(\tfrac{1}{2};\varepsilon_{1}\alpha_{1},\ldots,\varepsilon_{k}\alpha_{k}\right)|D|^{\tfrac{1}{2}\sum_{j=1}^{k}\alpha_{j}},
\end{equation}
and
\begin{equation}
f(s)=\zeta_{A}(1+s)\log q \ \ \ \ \ \ \mathrm{and \ so} \ \ \ \ \ f(\alpha_{i}+\alpha_{j})=\zeta_{A}(1+\alpha_{i}+\alpha_{j})\log q
\end{equation}
we have that $f(s)$ has a simple pole at $s=0$ with residue $1$. 

Denoting
\begin{equation}
K(\alpha_{1},\ldots,\alpha_{k})=F(\alpha_{1},\ldots,\alpha_{k})\prod_{1\leq i\leq j\leq k}f(\alpha_{i}+\alpha_{j}),
\end{equation}
we can write \eqref{eq:eqteste} as
\begin{equation}
\label{eq:momentos}
\Bigg(\sum_{D\in\mathcal{H}_{2g+1,q}}\frac{\prod_{j=1}^{k}X(\tfrac{1}{2}+\alpha_{j})|D|^{-\tfrac{1}{2}\sum_{j=1}^{k}\alpha_{j}}}{(\log q)^{k(k+1)/2}}\sum_{\varepsilon_{j}=\pm1}K(\varepsilon_{1}\alpha_{1},\ldots,\varepsilon_{k}\alpha_{k})\Bigg)(1+o(1)),
\end{equation}
and now we can use Lemma \ref{thm:concisesumsymplectic} to write
\begin{eqnarray}
& & \sum_{D\in\mathcal{H}_{2g+1,q}}\frac{\prod_{j=1}^{k}X(\tfrac{1}{2}+\alpha_{j})|D|^{-\tfrac{1}{2}\sum_{j=1}^{k}\alpha_{j}}}{(\log q)^{k(k+1)/2}}\frac{(-1)^{k(k-1)/2} } {(2\pi
i)^k} \frac{2^k}{k!} \oint \cdots \oint K(z_1,\ldots,z_k)\nonumber\\
& & \ \ \ \ \ \ \ \ \ \ \ \ \ \ \ \ \ \ \ \ \ \ \ \ \ \ \ \ \ \ \ \ \ \ \ \ \ \ \ \ \ \ \ \ \ \times \frac{\Delta(z_1^2,\ldots,z_k^2)^2 \prod_{j=1}^k z_j }
{\prod_{i=1}^k\prod_{j=1}^k (z_i-\alpha_j)(z_i+\alpha_j)}
\,dz_1\cdots dz_k+o(|D|)  \nonumber\\
& = &\label{eq:eqtest1} \sum_{D\in\mathcal{H}_{2g+1,q}}\prod_{j=1}^{k}X(\tfrac{1}{2}+\alpha_{j})|D|^{-\tfrac{1}{2}\sum_{j=1}^{k}\alpha_{j}}\frac{(-1)^{k(k-1)/2} } {(2\pi
i)^k} \frac{2^k}{k!}\oint \cdots \oint F(z_1,\ldots,z_k)\nonumber\\
& & \times\prod_{1\leq i\leq j\leq k}\zeta_{A}(1+z_{i}+z_{k})\frac{\Delta(z_1^2,\ldots,z_k^2)^2 \prod_{j=1}^k z_j }
{\prod_{i=1}^k\prod_{j=1}^k (z_i-\alpha_j)(z_i+\alpha_j)}
\,dz_1\cdots dz_k+o(|D|).
\end{eqnarray}

If we denote
\begin{equation}
K(z_{1},\ldots,z_{k})=F(z_{1},\ldots,z_{k})\prod_{1\leq i\leq j\leq k}\zeta_{A}(1+z_{i}+z_{k}),
\end{equation}

we have that \eqref{eq:eqtest1} becomes
\begin{multline}
\sum_{D\in\mathcal{H}_{2g+1,q}}\prod_{j=1}^{k}X(\tfrac{1}{2}+\alpha_{j})|D|^{-\tfrac{1}{2}\sum_{j=1}^{k}\alpha_{j}}\frac{(-1)^{k(k-1)/2} } {(2\pi
i)^k} \frac{2^k}{k!}\oint \cdots \oint K(z_1,\ldots,z_k)\\
\ \ \ \ \ \ \ \ \ \ \ \ \ \ \ \ \ \ \ \ \ \ \ \ \ \ \ \ \ \ \ \ \ \ \times \frac{\Delta(z_1^2,\ldots,z_k^2)^2 \prod_{j=1}^k z_j }
{\prod_{i=1}^k\prod_{j=1}^k (z_i-\alpha_j)(z_i+\alpha_j)}
\,dz_1\cdots dz_k+o(|D|),
\end{multline}
and if we denote
\begin{equation}
G(z_{1},\ldots,z_{k})=\prod_{j=1}^{k}X(\tfrac{1}{2}+z_{j})^{-1/2}A(\tfrac{1}{2};z_{1},\ldots,z_{k})\prod_{1\leq i\leq j\leq k}\zeta_{A}(1+z_{i}+z_{j})
\end{equation}
we have that the equation above is
\begin{multline}
\sum_{D\in\mathcal{H}_{2g+1,q}}\prod_{j=1}^{k}X(\tfrac{1}{2}+\alpha_{j})|D|^{-\tfrac{1}{2}\sum_{j=1}^{k}\alpha_{j}}\frac{(-1)^{k(k-1)/2} } {(2\pi
i)^k} \frac{2^k}{k!}\oint \cdots \oint G(z_1,\ldots,z_k)\\
\ \ \ \ \ \ \ \ \ \ \ \ \\ \ \ \ \ \ \ \  \ \ \ \ \ \ \ \ \ \ \ \ \times |D|^{\tfrac{1}{2}\sum_{j=1}^{k}z_{j}}\frac{\Delta(z_1^2,\ldots,z_k^2)^2 \prod_{j=1}^k z_j }
{\prod_{i=1}^k\prod_{j=1}^k (z_i-\alpha_j)(z_i+\alpha_j)}
\,dz_1\cdots dz_k+o(|D|).
\end{multline}

Now calling

\begin{multline}
Q_{k}(x)=\frac{(-1)^{k(k-1)/2} } {(2\pi
i)^k} \frac{2^k}{k!}\\
\times\oint \cdots \oint G(z_1,\ldots,z_k)\frac{\Delta(z_1^2,\ldots,z_k^2)^2 \prod_{j=1}^k z_j }
{\prod_{i=1}^k\prod_{j=1}^k (z_i-\alpha_j)(z_i+\alpha_j)}q^{\tfrac{x}{2}\sum_{j=1}^{k}z_{j}}dz_1\cdots dz_k,
\end{multline}
and setting $\alpha_{j}=0$, we have arrived at the formulae given in Conjecture \ref{thm:myconjecture}. 

%\begin{remark}
%Note that \eqref{eq:eq5.1.11} is the function field analogue of the formula (1.5.11) in \cite{CFKRS}. 
%since it corresponds to the $k^{\mathrm{th}}$--moments of quadratic Dirichlet $L$--functions for the rational function field $\mathbb{F}_{q}(T)$.
%\end{remark}

\section{Some conjectural formulae for moments of $L$-functions in the Hyperelliptic Ensemble}

In this section we use Conjecture \ref{thm:myconjecture} to obtain explicit conjectural formulae for the first few moments of quadratic Dirichlet $L$--functions over function fields.

\subsection{First Moment}

We will use Conjecture \ref{thm:myconjecture} to determine the asymptotics of the first moment ($k=1$) of our family of $L$--functions and compare with the main theorem of \cite{AK}. Specifically, we will specialize the formula in Conjecture \ref{thm:myconjecture} for $k=1$ to compute 
\begin{equation}
\sum_{D\in\mathcal{H}_{2g+1,q}}L(\tfrac{1}{2},\chi_{D})=\sum_{D\in\mathcal{H}_{2g+1,q}}Q_{1}(\log_{q}|D|)(1+o(1)),
\end{equation}
where $Q_{1}(x)$ is a polynomial of degree $1$, i.e., $Q_{1}(x)=ax+b$. This will be done using the contour integral formula for $Q_{k}(x)$. We have,

\begin{equation}
\label{eq:eq5.5.1}
Q_{1}(x)=\frac{1}{\pi i}\oint\frac{G(z_{1})\Delta(z_{1}^{2})^{2}}{z_{1}}q^{\tfrac{x}{2}z_{1}}dz_{1}
\end{equation}
where
\begin{equation}
G(z_{1})=A(\tfrac{1}{2};z_{1})X(\tfrac{1}{2}+z_{1})^{-1/2}\zeta_{A}(1+2z_{1}).
\end{equation}

Remembering that,
\begin{equation}
\Delta(z_{1},\ldots,z_{k})=\prod_{1\leq i<j\leq k}(z_{j}-z_{i})
\end{equation}
is the Vandermonde determinant we have that,
\begin{equation}
\Delta(z_{1}^{2})^{2}=1
\end{equation}
and
\begin{equation}
X(\tfrac{1}{2}+z_{1})^{-1/2}=q^{-z_{1}/2}.
\end{equation}

So \eqref{eq:eq5.5.1} becomes,

\begin{equation}
\label{eq:eqcontour}
%\frac{1}{\pi i}\oint\frac{A(\tfrac{1}{2};z_{1})X(\tfrac{1}{2}+z_{1})^{-1/2}\zeta_{A}(1+2z_{1})}{z_{1}}q^{\tfrac{x}{2}z_{1}}dz_{1}\\
Q_{1}(x)=\frac{1}{\pi i}\oint\frac{A(\tfrac{1}{2};z_{1})\zeta_{A}(1+2z_{1})q^{-z_{1}/2}}{z_{1}}q^{\tfrac{x}{2}z_{1}}dz_{1}.
\end{equation}

We also have that,
\begin{eqnarray}
& & A(\tfrac{1}{2};z_{1})=\prod_{\substack{P \ \mathrm{monic} \\ \mathrm{irreducible}}}\left(1-\frac{1}{|P|^{1+2z_{1}}}\right)\nonumber\\
& & \ \ \ \ \ \ \ \ \ \ \ \ \ \  \times \left(\frac{1}{2}\left(\left(1-\frac{1}{|P|^{1/2+z_{1}}}\right)^{-1}+\left(1+\frac{1}{|P|^{1/2+z_{1}}}\right)^{-1}\right)+\frac{1}{|P|}\right)\nonumber\\
& & \ \ \ \ \ \ \ \ \ \ \ \ \ \  \times \left(1+\frac{1}{|P|}\right)^{-1}.
\end{eqnarray}

Our goal is to compute the integral \eqref{eq:eqcontour} where the contour is a small circle around the origin, and for that we need to locate the poles of the integrand,

\begin{equation}
f(z_{1})=\frac{A(\tfrac{1}{2};z_{1})\zeta_{A}(1+2z_{1})q^{-z_{1}/2}}{z_{1}}q^{\tfrac{x}{2}z_{1}}.
\end{equation}

We note that $f(z_{1})$ has a pole of order $2$ at $z_{1}=0$. To compute the residue we expand $f(z_{1})$ as a Laurent series and pick up the coefficient of $1/z_{1}$. Expanding the numerator of $f(z_{1})$ around $z_{1}=0$ we have,

\begin{enumerate}
\item {} 
$$A(\tfrac{1}{2};z_{1})=A(\tfrac{1}{2},0)+A^{'}(\tfrac{1}{2},0)z_{1}+\frac{1}{2}A^{''}(\tfrac{1}{2},0)z_{1}^{2}+\cdots$$
\item {}
$$q^{-z_{1}/2}=1-\frac{1}{2}(\log q)z_{1}+\frac{1}{8}(\log q)z_{1}^{2}+\cdots$$
\item {}
$$q^{\tfrac{x}{2}z_{1}}=1+\frac{1}{2}(\log q)xz_{1}+\frac{1}{8}(\log^{2}q)x^{2}z_{1}^{2}+\cdots$$
\item {}
$$\zeta_{A}(1+2z_{1})=\frac{1}{2\log q}\frac{1}{z_{1}}+\frac{1}{2}+\frac{1}{6}(\log q)z_{1}-\frac{1}{90}(\log^{3}q)z_{1}^{3}+\cdots$$
\end{enumerate}

Hence we can write,

\begin{multline}
f(z_{1})=\left(\frac{A(\tfrac{1}{2};0)}{z_{1}}+A^{'}(\tfrac{1}{2};0)+\frac{A^{''}(\tfrac{1}{2};0)}{2}z_{1}+\cdots\right)\left(1-\frac{1}{2}(\log q)z_{1}+\frac{1}{8}(\log q)z_{1}^{2}+\cdots\right)\\
\times \left(1+\frac{1}{2}(\log q)xz_{1}+\frac{1}{8}(\log^{2}q)x^{2}z_{1}^{2}+\cdots\right)\ \ \ \ \ \ \ \ \ \ \ \ \ \ \ \ \ \ \ \ \ \ \ \ \ \ \ \ \ \ \ \ \\
\times\left(\frac{1}{2\log q}\frac{1}{z_{1}}+\frac{1}{2}+\frac{1}{6}(\log q)z_{1}-\frac{1}{90}(\log^{3}q)z_{1}^{3}+\cdots\right).\ \ \ \ \ \ \ \ \ \ \ \ \ \ \ \ \ \ \ \ \ \ \ \ \ \ \ \ \ \ \
\end{multline}

Multiplying the above expression we identify the coefficient of $1/z_{1}$. Therefore

\begin{equation}
\label{eq:eq5.5.5}
\mathrm{Res}_{z_{1}=0}f(z_{1})=\frac{1}{2}A(\tfrac{1}{2};0)-\frac{1}{4}A(\tfrac{1}{2};0)+\frac{1}{4}A(\tfrac{1}{2};0)x+\frac{1}{2\log q}A^{'}(\tfrac{1}{2};0). 
\end{equation}

We find, after some straightforward calculations, that:

\begin{equation}
A(\tfrac{1}{2};0)=P(1)=\prod_{\substack{P \ \mathrm{monic} \\ \mathrm{irreducible}}}\left(1-\frac{1}{(|P|+1)|P|}\right)
\end{equation}
and
\begin{equation}
A^{'}(\tfrac{1}{2};0)=A(\tfrac{1}{2};0)(2\log q)\sum_{\substack{P \ \mathrm{monic} \\ \mathrm{irreducible}}}\frac{\mathrm{deg}(P)}{|P|(|P|+1)-1}
\end{equation}
and so \eqref{eq:eq5.5.5} is
\begin{equation}
\mathrm{Res}_{z_{1}=0}f(z_{1})=\frac{1}{4}P(1)+\frac{1}{4}P(1)x+P(1)\sum_{\substack{P \ \mathrm{monic} \\ \mathrm{irreducible}}}\frac{\mathrm{deg}(P)}{|P|(|P|+1)-1}. 
\end{equation}

Hence we have that,

\begin{multline}
\frac{1}{\pi i}\oint\frac{A(\tfrac{1}{2};z_{1})\zeta_{A}(1+2z_{1})q^{-z_{1}/2}}{z_{1}}q^{\tfrac{x}{2}z_{1}}dz_{1}\\
=\frac{1}{\pi i}2\pi i\left(\frac{1}{4}P(1)+\frac{1}{4}P(1)x+P(1)\sum_{\substack{P \ \mathrm{monic} \\ \mathrm{irreducible}}}\frac{\mathrm{deg}(P)}{|P|(|P|+1)-1}\right)\\
=\frac{1}{2}P(1)+\frac{1}{2}P(1)x+2P(1)\sum_{\substack{P \ \mathrm{monic} \\ \mathrm{irreducible}}}\frac{\mathrm{deg}(P)}{|P|(|P|+1)-1}.\ \ \ \ \ \ \ \ \ \ \ \ \ \ \ \ \ \ \ \ \ \ \ \ \ \ \ \
\end{multline}

So, 
\begin{equation}
Q_{1}(x)=\frac{1}{2}P(1)\left\{x+1+4\sum_{\substack{P \ \mathrm{monic} \\ \mathrm{irreducible}}}\frac{\mathrm{deg}(P)}{|P|(|P|+1)-1}\right\}.
\end{equation}

We therefore have that

\begin{eqnarray}
& & \sum_{D\in\mathcal{H}_{2g+1,q}}L(\tfrac{1}{2},\chi_{D})=\sum_{D\in\mathcal{H}_{2g+1,q}}Q_{1}(\log_{q}|D|)(1+o(1))\nonumber\\
& = & \sum_{D\in\mathcal{H}_{2g+1,q}}\frac{1}{2}P(1)\left\{\log_{q}|D|+1+4\sum_{\substack{P \ \mathrm{monic} \\ \mathrm{irreducible}}}\frac{\mathrm{deg}(P)}{|P|(|P|+1)-1}\right\}(1+o(1))\nonumber\\
& = & \frac{1}{2}P(1)\left\{\log_{q}|D|+1+4\sum_{\substack{P \ \mathrm{monic} \\ \mathrm{irreducible}}}\frac{\mathrm{deg}(P)}{|P|(|P|+1)-1}\right\}\sum_{D\in\mathcal{H}_{2g+1,q}}1+o(|D|)\nonumber\\
& = & \frac{P(1)}{2\zeta_{A}(2)}|D|\left\{\log_{q}|D|+1+4\sum_{\substack{P \ \mathrm{monic} \\ \mathrm{irreducible}}}\frac{\mathrm{deg}(P)}{|P|(|P|+1)-1}\right\}+o(|D|).
\end{eqnarray}

If we compare the main theorem of \cite{AK} with the conjecture we note that the main term and the principal lower order terms are the same. Hence the main theorem of \cite{AK} proves our conjecture with an error $O(|D|^{3/4+\varepsilon})$ when $k=1$.

\subsection{Second Moment}

For the second moment, Conjecture \ref{thm:myconjecture} asserts that

\begin{equation}
\sum_{D\in\mathcal{H}_{2g+1,q}}L(\tfrac{1}{2},\chi_{D})^{2}=\sum_{D\in\mathcal{H}_{2g+1,q}}Q_{2}(\log_{q}|D|)(1+o(1)),
\end{equation}
where

\begin{equation}
Q_{2}(x)=\frac{(-1)2^{2}}{2!}\frac{1}{(2\pi i)^{2}}\oint\oint\frac{G(z_{1},z_{2})\Delta(z_{1}^{2},z_{2}^{2})^{2}}{z_{1}^{3}z_{2}^{3}}q^{\tfrac{x}{2}(z_{1}+z_{2})}dz_{1}dz_{2}.
\end{equation}

%Denoting by $A_{z_{i}}(0, ... 0)$ meaning the derivative with respect to $z_{i}$ at the point $(0,\ldots,0)$ of the function of $A_{k}(z_{1},\ldots,z_{k})$ given in \eqref{eq:A}, we have that,

We denote by $A_{j}$ the partial derivative, evaluated at zero, of the function $A(\tfrac{1}{2};z_{1},\ldots,z_{k})$ with respect to $j$th variable, with repeated indices denoting higher derivatives.  So, for example 

\begin{equation}
A_{112}(0, 0, \ldots, 0)=\frac{\partial^{2}}{\partial z_{1}^{2}}\frac{\partial}{\partial z_{2}}A(\tfrac{1}{2}; z_{1},z_{2}, \ldots, z_{k})\bigg|_{z_{1}=z_{2}=\cdots=z_{k}=0}.
\end{equation}
We then have that,

\begin{multline}
\oint\oint\frac{G(z_{1},z_{2})\Delta(z_{1}^{2},z_{2}^{2})^{2}}{z_{1}^{3}z_{2}^{3}}q^{\tfrac{x}{2}(z_{1}+z_{2})}dz_{1}dz_{2}\\
=(2\pi i)^2\Bigg[-\frac{1}{48(\log q)^{3}}\Bigg((6+11x+6x^{2}+x^{3})A(0,0)(\log q)^{3}\ \ \ \ \ \ \ \ \ \ \ \ \ \ \ \ \ \ \ \ \ \ \ \ \ \ \ \ \ \ \ \ \\
+(11+12x+3x^{2})(\log q)^{2}(A_{2}(0,0)+A_{1}(0,0))+12(2+x)(\log q)A_{12}(0,0)\ \ \ \ \ \ \ \ \ \ \ \ \ \\
-2(A_{222}(0,0)-3A_{122}(0,0)-3A_{112}(0,0)+A_{111}(0,0))\Bigg)\Bigg].\ \ \ \ \ \ \ \ \ \ \ \ \ \ \ \ \ \ \ \ \ \ \ \ \ \ \ \ \ \ \ \ \ \ \ \ 
\end{multline}

%\begin{multline}
%\oint\oint\frac{G(z_{1},z_{2})\Delta(z_{1}^{2},z_{2}^{2})^{2}}{z_{1}^{3}z_{2}^{3}}q^{\tfrac{x}{2}(z_{1}+z_{2})}dz_{1}dz_{2}\\
%=(2\pi i)^2\Bigg[\frac{1}{48(\log q)^{3}}\Bigg(-6A(0,0)(\log q)^{3}-11xA(0,0)(\log q)^{3}-6x^{2}A(0,0)(\log q)^{3}\\
%-x^{3}A(0,0)(\log q)^{3}-11(\log q)^{2}A_{z_{2}}(0,0)-12x(\log q)^{2}A_{z_{2}}(0,0)-3x^{2}(\log q)^{2}A_{z_{1}}(0,0)\\
%+2A_{z_{2}z_{2}z_{2}}(0,0)-11(\log q)^{2}A_{z_{1}}(0,0)-12x(\log q)^{2}A_{z_{1}}(0,0)-3x^{2}(\log q)^{2}A_{z_{1}}(0,0)\\
%-24(\log q)A_{z_{1}z_{2}}(0,0)-12x(\log q)A_{z_{1}z_{2}}(0,0)-6A_{z_{1}z_{2}z_{2}}(0,0)-6A_{z_{2}z_{2}z_{1}}(0,0)+2A_{z_{1}z_{1}z_{1}}(0,0)\Bigg)\Bigg].
%\end{multline}

Hence the leading order asymptotic for the second moment for this family of $L$--functions can be written, conjecturally, as

\begin{equation}
\sum_{D\in\mathcal{H}_{2g+1,q}}L(\tfrac{1}{2},\chi_{D})^{2}\sim\frac{1}{24\zeta_{A}(2)}A(\tfrac{1}{2}; 0,0)|D|(\log_{q}|D|)^{3},
\end{equation}
when $g\rightarrow\infty$, where
\begin{equation}
A(\tfrac{1}{2}; 0,0)=\prod_{\substack{P \ \mathrm{monic}\\ \mathrm{irreducible}}}\left(1-\frac{4|P|^{2}-3|P|+1}{|P|^{4}+|P|^{3}}\right).
\end{equation}

\subsection{Third Moment}

For the third moment, our conjecture states that:

\begin{equation}
\sum_{D\in\mathcal{H}_{2g+1,q}}L(\tfrac{1}{2},\chi_{D})^{3}=\sum_{D\in\mathcal{H}_{2g+1,q}}Q_{3}(\log_{q}|D|)(1+o(1)),
\end{equation}
where

\begin{equation}
Q_{3}(x)=\frac{(-1)^{3}2^{3}}{3!}\frac{1}{(2\pi i)^{3}}\oint\oint\oint\frac{G(z_{1},z_{2},z_{3})\Delta(z_{1}^{2},z_{2}^{2},z_{3}^{2})^{2}}{z_{1}^{5}z_{2}^{5}z_{3}^{5}}q^{\tfrac{x}{2}(z_{1}+z_{2}+z_{3})}dz_{1}dz_{2}dz_{3}.
\end{equation}

Computing the triple contour integral with the help of the symbolic manipulation software MATHEMATICA we obtain %and using the notation that $A_{j}(\tfrac{1}{2};z_{1},\ldots,z_{k})=A_{j}(z_{1},\ldots,z_{k})$ we obtain

\begin{multline}
\oint\oint\oint\frac{G(z_{1},z_{2},z_{3})\Delta(z_{1}^{2},z_{2}^{2},z_{3}^{2})^{2}}{z_{1}^{5}z_{2}^{5}z_{3}^{5}}q^{\tfrac{x}{2}(z_{1}+z_{2}+z_{3})}dz_{1}dz_{2}dz_{3}\\
=(2\pi i)^{3}\Bigg[-\frac{1}{11520(\log q)^{6}}\Bigg(3(3+x)^{2}(40+78x+49x^{2}+12x^{3}+x^{4})A(0,0,0)(\log q)^{6}\ \ \ \ \ \ \ \ \ \ \ \ \ \ \ \ \ \ \ \ \ \ \ \ \ \ \ \ \ \ \ \ \ \ \ \ \ \ \ \ \ \ \ \ \ \ \ \ \ \ \ \ \ \ \ \  \ \ \ \ \ \ \ \ \\
+4(471+949x+720x^{2}+260x^{3}+45x^{4}+3x^{5})(\log q)^{5}(A_{3}(0,0,0)+A_{2}(0,0,0)\ \ \ \ \ \ \ \ \ \ \ \ \ \ \ \ \ \ \ \ \ \ \ \ \ \ \ \ \ \ \ \ \ \ \ \ \ \ \ \ \ \ \ \ \ \ \ \ \ \ \ \ \ \ \ \  \ \ \ \ \ \ \ \ \\
+A_{1}(0,0,0))+4(949+1440x+780x^{2}+180x^{3}+15x^{4})(\log q)^{4}(A_{23}(0,0,0)\ \ \ \ \ \ \ \ \ \ \ \ \ \ \ \ \ \ \ \ \ \ \ \ \ \ \ \ \ \ \ \ \ \ \ \ \ \ \ \ \ \ \ \ \ \ \ \ \ \ \ \ \ \ \ \  \ \ \ \ \ \ \ \ \\
+A_{13}(0,0,0)+A_{12}(0,0,0))-10(24+26x+9x^{2}+x^{3})(\log q)^{3}(2A_{333}(0,0,0)\ \ \ \ \ \ \ \ \ \ \ \ \ \ \ \ \ \ \ \ \ \ \ \ \ \ \ \ \ \ \ \ \ \ \ \ \ \ \ \ \ \ \ \ \ \ \ \ \ \ \ \ \ \ \ \  \ \ \ \ \ \ \ \ \\
-3A_{233}(0,0,0)-3A_{223}(0,0,0)+2A_{222}(0,0,0)-3A_{133}(0,0,0)-36A_{123}(0,0,0)\ \ \ \ \ \ \ \ \ \ \ \ \ \ \ \ \ \ \ \ \ \ \ \ \ \ \ \ \ \ \ \ \ \ \ \ \ \ \ \ \ \ \ \ \ \ \ \ \ \ \ \ \ \ \ \  \ \ \ \ \ \ \ \ \\
-3A_{122}(0,0,0)-3A_{113}(0,0,0)-3A_{112}(0,0,0)+2A_{111}(0,0,0))\ \ \ \ \ \ \ \ \ \ \ \ \ \ \ \ \ \ \ \ \ \ \ \ \ \ \ \ \ \ \ \ \ \ \ \ \ \ \ \ \ \ \ \ \ \ \ \ \ \ \ \ \ \ \ \  \ \ \ \ \ \ \ \ \\
-20(26+18x+3x^{2})(\log q)^{2}(A_{2333}(0,0,0)+A_{2223}(0,0,0)\ \ \ \ \ \ \ \ \ \ \ \ \ \ \ \ \ \ \ \ \ \ \ \ \ \ \ \ \ \ \ \ \ \ \ \ \ \ \ \ \ \ \ \ \ \ \ \ \ \ \ \ \ \ \ \  \ \ \ \ \ \ \ \ \\
+A_{1333}(0,0,0)-6A_{1233}(0,0,0)-6A_{1223}(0,0,0)+A_{1222}(0,0,0)-6A_{1123}(0,0,0)\ \ \ \ \ \ \ \ \ \ \ \ \ \ \ \ \ \ \ \ \ \ \ \ \ \ \ \ \ \ \ \ \ \ \ \ \ \ \ \ \ \ \ \ \ \ \ \ \ \ \ \ \ \ \ \  \ \ \ \ \ \ \ \ \\
+A_{1113}(0,0,0)+A_{1112}(0,0,0))+6(3+x)(\log q)(2A_{33333}(0,0,0)-5A_{23333}(0,0,0)\ \ \ \ \ \ \ \ \ \ \ \ \ \ \ \ \ \ \ \ \ \ \ \ \ \ \ \ \ \ \ \ \ \ \ \ \ \ \ \ \ \ \ \ \ \ \ \ \ \ \ \ \ \ \ \  \ \ \ \ \ \ \ \ \\
-10A_{22333}(0,0,0)-10A_{22233}(0,0,0)-5A_{22223}(0,0,0)+2A_{22222}(0,0,0)\ \ \ \ \ \ \ \ \ \ \ \ \ \ \ \ \ \ \ \ \ \ \ \ \ \ \ \ \ \ \ \ \ \ \ \ \ \ \ \ \ \ \ \ \ \ \ \ \ \ \ \ \ \ \ \  \ \ \ \ \ \ \ \ \\
-5A_{13333}(0,0,0)+60A_{12233}(0,0,0)-5A_{12222}(0,0,0)-10A_{11333}(0,0,0)\ \ \ \ \ \ \ \ \ \ \ \ \ \ \ \ \ \ \ \ \ \ \ \ \ \ \ \ \ \ \ \ \ \ \ \ \ \ \ \ \ \ \ \ \ \ \ \ \ \ \ \ \ \ \ \  \ \ \ \ \ \ \ \ \\
+60A_{11233}(0,0,0)+60A_{11223}(0,0,0)-10A_{11222}(0,0,0)-10A_{11133}(0,0,0)\ \ \ \ \ \ \ \ \ \ \ \ \ \ \ \ \ \ \ \ \ \ \ \ \ \ \ \ \ \ \ \ \ \ \ \ \ \ \ \ \ \ \ \ \ \ \ \ \ \ \ \ \ \ \ \  \ \ \ \ \ \ \ \ \\
-10A_{11122}(0,0,0)-5A_{11113}(0,0,0)-5A_{11112}(0,0,0)+2A_{11111}(0,0,0))\ \ \ \ \ \ \ \ \ \ \ \ \ \ \ \ \ \ \ \ \ \ \ \ \ \ \ \ \ \ \ \ \ \ \ \ \ \ \ \ \ \ \ \ \ \ \ \ \ \ \ \ \ \ \ \  \ \ \ \ \ \ \ \\
+4(3A_{233333}(0,0,0)-20A_{222333}(0,0,0)+3A_{222223}(0,0,0)+3A_{222223}(0,0,0)\ \ \ \ \ \ \ \ \ \ \ \ \ \ \ \ \ \ \ \ \ \ \ \ \ \ \ \ \ \ \ \ \ \ \ \ \ \ \ \ \ \ \ \ \ \ \ \ \ \ \ \ \ \ \ \  \ \ \ \ \ \ \ \ \ \ \ \ \ \ \ \ \ \ \ \ \ \ \ \ \ \ \ \ \ \ \ \ \ \ \ \ \ \ \ \ \ \ \ \ \ \ \ \ \ \ \ \ \ \ \ \ \ \ \ \ \ \ \ \ \ \ \ \ \ \ \ \ \ \ \ \ \ \ \ \ \ \ \ \ \ \ \  \ \ \ \ \ \ \ \ \\
-30A_{123333}(0,0,0)+30A_{122333}(0,0,0)+30A_{122233}(0,0,0)-30A_{122223}(0,0,0)\ \ \ \ \ \ \ \ \ \ \ \ \ \ \ \ \ \ \ \ \ \ \ \ \ \ \ \ \ \ \ \ \ \ \ \ \ \ \ \ \ \ \ \ \ \ \ \ \ \ \ \ \ \ \ \  \ \ \ \ \ \ \ \ \\
+3A_{122222}(0,0,0)+30A_{112333}(0,0,0)+30A_{112223}(0,0,0)-20A_{111333}(0,0,0)\ \ \ \ \ \ \ \ \ \ \ \ \ \ \ \ \ \ \ \ \ \ \ \ \ \ \ \ \ \ \ \ \ \ \ \ \ \ \ \ \ \ \ \ \ \ \ \ \ \ \ \ \ \ \ \  \ \ \ \ \ \ \ \ \\
+30A_{111233}(0,0,0)+30A_{111223}(0,0,0)-20A_{111222}(0,0,0)-30A_{111123}\ \ \ \ \ \ \ \ \ \ \ \ \ \ \ \ \ \ \ \ \ \ \ \ \ \ \ \ \ \ \ \ \ \ \ \ \ \ \ \ \ \ \ \ \ \ \ \ \ \ \ \ \ \ \ \  \ \ \ \ \ \ \ \ \\
+3A_{111113}(0,0,0)+3A_{11112}(0,0,0))\Bigg)\Bigg].\ \ \ \ \ \ \ \ \ \ \ \ \ \  \ \ \ \ \ \ \ \ \ \ \ \ \ \ \ \ \ \ \ \ \ \ \ \ \ \ \ \ \ \ \ \ \ \ \ \  \ \ \ \ \ \ \ \ \ \ \ \ \ \ \ \ \ \ \ \ \ \ \ \  \ \ \ \ \ \ \ \ \ \ \ \ \ \ \ \ \ \ \ \ \ \ \ \  \ \ \
\end{multline}

And so, identifying the coefficient of $x^{6}$, we conjecture 

\begin{equation}
\sum_{D\in\mathcal{H}_{2g+1,q}}L(\tfrac{1}{2},\chi_{D})^{3}\sim\frac{1}{2880\zeta_{A}(2)}A(\tfrac{1}{2};0,0,0)|D|(\log_{q}|D|)^{6},
\end{equation}
as $g\rightarrow\infty$, where
\begin{equation}
A(\tfrac{1}{2};0,0,0)=\prod_{\substack{P \ \mathrm{monic}\\ \mathrm{irreducible}}}\left(1-\frac{12|P|^{5}-23|P|^{4}+23|P|^{3}-15|P|^{2}+6|P|-1}{|P|^{6}(|P|+1)}\right).
\end{equation}

\subsection{Leading Order for general $k$.}In this section we will show how to obtain an explicit conjecture for the leading order asymptotic of the moments for a general integer $k$. The calculations presented here follow closely those presented in \cite{KO}. The main result is the following conjecture:

\begin{thm}
\label{conj:assintotico}
Conditional on Conjecture \ref{thm:myconjecture} we have that as $g\rightarrow\infty$ the following holds
\begin{equation}
\sum_{D\in\mathcal{H}_{2g+1,q}}L(\tfrac{1}{2},\chi_{D})^{k}\sim\frac{|D|}{\zeta_{A}(2)}(\log_{q}|D|)^{k(k+1)/2}A(\tfrac{1}{2};0,\ldots,0)\prod_{j=1}^{k}\frac{j!}{(2j)!}.
\end{equation} 
\end{thm}

To establish the above theorem we will first prove the following lemma.

\begin{lem}
Suppose $F$ is a symmetric function of $k$ variables, regular near $(0,\ldots,0)$ and $f(s)$ has a simple pole of residue $1$ at $s=0$ and is otherwise analytic in a neighborhood of $s=0$. Let
\begin{equation}
K(|D|;w_{1},\ldots,w_{k})=\sum_{\varepsilon_{j}=\pm1}e^{\tfrac{1}{2}\log|D|\sum_{j=1}^{k}\varepsilon_{j}w_{j}}F(\varepsilon_{1}w_{1},\ldots,\varepsilon_{j}w_{j})\prod_{1\leq i\leq j\leq k}f(\varepsilon_{i}w_{i}+\varepsilon_{j}w_{j})
\end{equation}
and define $I(|D|,k;w=0)$ to be the value of $K$ when $w_{1},\ldots,w_{k}=0$. We have that,
\begin{equation}
I(|D|,k;0)\sim(\tfrac{1}{2}\log|D|)^{k(k+1)/2}F(0,\ldots,0)2^{k(k+1)/2}\left(\prod_{j=1}^{k}\frac{j!}{(2j)!}\right).
\end{equation}
\end{lem}
\begin{proof}
We begin by defining the following function

\begin{equation}
G(|D|;w_{1},\ldots,w_{k})=e^{\tfrac{1}{2}\log|D|\sum_{j=1}^{k}w_{j}}F(w_{1},\ldots,w_{k})\prod_{1\leq i\leq j\leq k}f(w_{i}+w_{j}).
\end{equation}

So by Lemma 2.5.2 of \cite{CFKRS} we have,
\begin{multline}
\label{eq:contour2}
\sum_{\varepsilon_{j}=\pm1}G(|D|;\varepsilon_{1}w_{1},\ldots,\varepsilon_{k}w_{k}) \\
=\frac{(-1)^{k(k-1)/2}}{(2\pi i)^{k}}\frac{2^{k}}{k!}\oint\cdots\oint G(|D|;z_{1},\ldots,z_{k})\frac{\Delta(z_{1}^{2},\ldots,z_{k}^{2})^{2}\prod_{j=1}^{k}z_{j}}{\prod_{i=1}^{k}\prod_{j=1}^{k}(z_{j}-w_{j})(z_{i}-w_{j})}dz_{1}\ldots dz_{k}.
\end{multline}
We will analyze this integral as $w_{j}\rightarrow0$. It follows from \eqref{eq:contour2} that
\begin{multline}
I(|D|,k;0)\\
=\frac{(-1)^{k(k-1)/2}}{(2\pi i)^{k}}\frac{2^{k}}{k!}\oint\cdots\oint G(|D|;z_{1},\ldots,z_{k})\frac{\Delta(z_{1}^{2},\ldots,z_{k}^{2})^{2}\prod_{j=1}^{k}z_{j}}{\prod_{j=1}^{k}z_{j}^{2k}}dz_{1}\ldots dz_{k}.
\end{multline}
We expand $G(|D|;z_{1},\ldots,z_{k})$ and make the following variable change $z_{j}=\frac{2v_{j}}{\log|D|}$ which provides us with
\begin{multline}
I(|D|,k;0)=\left(\frac{1}{2}\log|D|\right)^{k(k+1)/2}\\
\times\frac{(-1)^{k(k-1)/2}}{(2\pi i)^{k}}\frac{1}{k!}\oint\cdots\oint e^{\sum_{j=1}^{k}v_{j}}F(2v_{1}/\log|D|,\ldots,2v_{k}/\log|D|)\ \ \ \ \ \ \ \ \ \ \ \ \ \ \ \ \ \ \ \ \  \ \ \ \ \ \ \ \ \ \ \ \ \ \ \ \ \ \ \ \ \ \ \ \  \ \ \ \\
\times\prod_{1\leq i<j\leq k}f\left(\frac{2}{\log|D|}(v_{i}+v_{j})\right)\left(\frac{2}{\log|D|}(v_{i}+v_{j})\right)\prod_{j=1}^{k}f\left(\frac{2}{\log|D|}(2v_{j})\right)\left(\frac{2}{\log|D|}(2v_{j})\right)\ \ \ \ \ \ \ \ \ \ \ \ \ \ \ \ \ \ \ \ \  \ \ \ \\
\times\prod_{1\leq i<j\leq k}\frac{1}{v_{i}+v_{j}}\frac{\Delta(v_{1}^{2},\ldots,v_{k}^{2})^{2}}{\prod_{j=1}^{k}v_{j}^{2k}}dv_{1}\ldots dv_{k}.\ \ \ \ \ \ \ \ \ \ \ \ \ \ \ \ \ \ \ \ \ \ \ \ \ \ \ \ \ \ \ \ \ \ \ \ \ \ \ \ \ \ \ \ \ \ \ \ \ \ \ \ \ \ \ \ \ \ \ \ \ \ \ \ \ \ \ \ \ \ \ \ \ \ \ \ \ \ \ \ \ \ \ \ \ \ \ \ \ \ \ \ \  \ \ \
\end{multline}
Letting $g\rightarrow\infty$ (i.e. $|D|\rightarrow\infty$) we have,
\begin{multline}
I(|D|,k;0)\sim\left(\frac{1}{2}\log|D|\right)^{k(k+1)/2}F(0,\ldots,0)\\
\times \frac{(-1)^{k(k-1)/2}}{(2\pi i)^{k}}\frac{1}{k!}\oint\cdots\oint e^{\sum_{j=1}^{k}v_{j}}\prod_{1\leq i<j\leq k}\frac{1}{v_{i}+v_{j}}\frac{\Delta(v_{1}^{2},\ldots,v_{k}^{2})^{2}}{\prod_{j=1}^{k}v_{j}^{2k}}dv_{1}\ldots dv_{k}.
\end{multline}
Using equation (3.36) from \cite{CFKRS2}, Lemma 2.5.2 from \cite{CFKRS}, and the result from \cite{KeS2} for the moments at the symmetry point for the symplectic ensemble completes the proof of the lemma.  
\end{proof}

Now we are ready to establish Theorem \ref{conj:assintotico}. Using \eqref{eq:momentos}, which is a conjectural formula, with $\alpha_{1},\ldots,\alpha_{k}=0$ and the lemma above we have that

\begin{multline}
\sum_{D\in\mathcal{H}_{2g+1,q}}L(\tfrac{1}{2},\chi_{D})^{k}\sim\sum_{D\in\mathcal{H}_{2g+1,q}}\frac{1}{(\log q)^{k(k+1)/2}}\\
\times\left(\frac{1}{2}\log|D|\right)^{k(k+1)/2}A(\tfrac{1}{2};0,\ldots,0)2^{k(k+1/2)}\prod_{j=1}^{k}\frac{j!}{(2j)!}.
\end{multline}

So as $g\rightarrow\infty$ we have the formula given in the conjecture.

\section{Ratios Conjecture for $L$--functions over Function Fields}

In this section we will present a natural generalization of Conjecture \ref{thm:myconjecture}: we give a heuristic for all of the main terms in the quotient of products of $L$--functions over function fields averaged over a family of hyperelliptic curves. The family of curves that we consider is the same as that considered above: curves of the form $C_{D}:y^{2}=D(x)$, where $D(x)\in\mathcal{H}_{2g+1,q}$. Essentially the goal is to adjust the recipe presented by Conrey, Farmer and Zirnbauer \cite{CFZ} for the case of quadratic Dirichlet $L$--functions over function fields.

Recall that in section 2 we introduced our family of $L$--functions. In particular if 
\begin{equation}
\mathcal{H}_{2g+1,q}=\left\{D \ \mathrm{monic}, \ D \ \mathrm{square-free}, \ \mathrm{deg}(D)=2g+1, \ D\in\mathbb{F}_{q}[x]\right\}
\end{equation}
the family $\mathcal{D}=\left\{L(s,\chi_{D}):D\in\mathcal{H}_{2g+1,q}\right\}$ is a symplectic family. We can make a conjecture which is the function field analogue of conjecture 5.2 in \cite{CFZ} for
\begin{equation}
\sum_{D\in\mathcal{H}_{2g+1,q}}\frac{\prod_{k=1}^{K}L(\tfrac{1}{2}+\alpha_{k},\chi_{D})}{\prod_{m=1}^{Q}L(\tfrac{1}{2}+\gamma_{m},\chi_{D})}.
\end{equation}

The main difficulty will be to identify and factor out the appropriate zeta factors (arithmetic factors) as was done in the previous section. We follow the recipe given in \cite[Section 5]{CFZ} and we will adapt the recipe for the function field setting when necessary.

The $L$--functions in the numerator are written as

\begin{equation}
\label{eq:funceqII}
L(s,\chi_{D})=\sum_{\substack{n \ \mathrm{monic} \\ \mathrm{deg}(n)\leq g}}\frac{\chi_{D}(n)}{|n|^{s}}+\mathcal{X}_{D}(s)\sum_{\substack{n \ \mathrm{monic} \\ \mathrm{deg}(n)\leq g-1}}\frac{\chi_{D}(n)}{|n|^{1-s}}
\end{equation}
and those in denominator are expanded into series
\begin{equation}
\frac{1}{L(s,\chi_{D})}=\prod_{\substack{P \ \mathrm{monic} \\ \mathrm{irreducible}}}\left(1-\frac{\chi_{D}(P)}{|P|^{s}}\right)=\sum_{n \ \mathrm{monic}}\frac{\mu(n)\chi_{D}(n)}{|n|^{s}}
\end{equation}
with $\mu(n)$ and $\chi_{D}(n)$ defined in Section 2.

In the numerator we will again replace $L(s,\chi_{D})$ with $Z_{L}(s,\chi_{D})$ and so the quantity that we will apply the recipe to is

\begin{eqnarray}
\label{eq:eq6.6}
& & \sum_{D\in\mathcal{H}_{2g+1,q}}\frac{\prod_{k=1}^{K}Z_{L}(\tfrac{1}{2}+\alpha_{k},\chi_{D})}{\prod_{m=1}^{Q}L(\tfrac{1}{2}+\gamma_{m},\chi_{D})}\\
& = &\sum_{D\in\mathcal{H}_{2g+1,q}}Z_{L}(\tfrac{1}{2}+\alpha_{1},\chi_{D})\ldots Z_{L}(\tfrac{1}{2}+\alpha_{k},\chi_{D})\sum_{\substack{h_{1},\ldots,h_{Q} \\ h_{i} \ \mathrm{monic}}}\frac{\mu(h_{1})\ldots\mu(h_{Q})\chi_{D}(h_{1}\ldots h_{Q})}{\prod_{m=1}^{Q}|h_{m}|^{\tfrac{1}{2}+\gamma_{m}}}.\nonumber
\end{eqnarray}
We have that,
\begin{multline}\label{eq:6.3.7}
Z_{L}(\tfrac{1}{2}+\alpha_{1},\chi_{D})\ldots Z_{L}(\tfrac{1}{2}+\alpha_{k},\chi_{D})\\
=\sum_{\epsilon_{k}\in\{-1,1\}^{K}}\prod_{k=1}^{K}\mathcal{X}_{D}(\tfrac{1}{2}+\epsilon_{k}\alpha_{k})^{-1/2}\sum_{\substack{m_{1},\ldots,m_{K} \\ m_{j} \ \mathrm{monic}}}\frac{\chi_{D}(m_{1}\ldots m_{k})}{|m_{k}|^{\tfrac{1}{2}+\epsilon_{k}\alpha_{k}}},
\end{multline}
and so, \eqref{eq:eq6.6} becomes

\begin{equation}
\sum_{D\in\mathcal{H}_{2g+1,q}}\sum_{\epsilon_{k}\in\{-1,1\}^{K}}\prod_{k=1}^{K}\mathcal{X}_{D}(\tfrac{1}{2}+\epsilon_{k}\alpha_{k})^{-1/2}\sum_{\substack{m_{1},\ldots,m_{K} \\ h_{1},\ldots,h_{Q} \\ m_{j},h_{i} \ \mathrm{monic}}}\frac{\prod_{m=1}^{Q}\mu(h_{m})\chi_{D}(m_{1}\ldots m_{K})\chi_{D}(h_{1}\ldots h_{Q})}{\prod_{k=1}^{K}|m_{k}|^{\tfrac{1}{2}+\epsilon_{k}\alpha_{k}}\prod_{m=1}^{Q}|h_{m}|^{\tfrac{1}{2}+\gamma_{m}}}.
\end{equation}

Now, following the recipe we average the summand over fundamental discriminants $D\in\mathcal{H}_{2g+1,q}$

\begin{eqnarray}
& & \lim_{\mathrm{deg}(D)\rightarrow\infty}\sum_{\epsilon_{k}\in\{-1,1\}^{K}}\prod_{k=1}^{K}X_{D}(\tfrac{1}{2}+\epsilon_{k}\alpha_{k})^{-1/2}\sum_{\substack{m_{1},\ldots,m_{K} \\ h_{1},\ldots,h_{Q} \\ m_{j},h_{i} \ \mathrm{monic}}}\frac{\prod_{m=1}^{Q}\mu(h_{m})\langle\chi_{D}(\prod_{k=1}^{K}m_{k}\prod_{m=1}^{Q}h_{m})\rangle}{\prod_{k=1}^{K}|m_{k}|^{\tfrac{1}{2}+\epsilon_{k}\alpha_{k}}\prod_{m=1}^{Q}|h_{m}|^{\tfrac{1}{2}+\gamma_{m}}}\nonumber\\
& = & \sum_{\epsilon_{k}\in\{-1,1\}^{K}}\prod_{k=1}^{K}X_{D}(\tfrac{1}{2}+\epsilon_{k}\alpha_{k})^{-1/2}\sum_{\substack{m_{1},\ldots,m_{K} \\ h_{1},\ldots,h_{Q} \\ m_{j},h_{i} \ \mathrm{monic}}}\frac{\prod_{m=1}^{Q}\mu(h_{m})\delta\left(\prod_{k=1}^{K}m_{k}\prod_{m=1}^{Q}h_{m}\right)}{\prod_{k=1}^{K}|m_{k}|^{\tfrac{1}{2}+\epsilon_{k}\alpha_{k}}\prod_{m=1}^{Q}|h_{m}|^{\tfrac{1}{2}+\gamma_{m}}}\nonumber\\
\end{eqnarray}
where $\displaystyle{\delta(n)=\prod_{\substack{P \ \mathrm{monic} \\ P \ \mathrm{irreducible} \\ P\mid n}}\left(1+\frac{1}{|P|}\right)^{-1}} \text{if $n$ is a square and is $0$ otherwise.}$

So, using the same notation as in \cite{CFZ}

\begin{equation}
G_{\mathcal{D}}(\alpha;\gamma)=\sum_{\substack{m_{1},\ldots,m_{K} \\ h_{1},\ldots,h_{Q} \\ m_{j},h_{i} \ \mathrm{monic}}}\frac{\prod_{m=1}^{Q}\mu(h_{m})\delta\left(\prod_{k=1}^{K}m_{k}\prod_{m=1}^{Q}h_{m}\right)}{\prod_{k=1}^{K}|m_{k}|^{\tfrac{1}{2}+\alpha_{k}}\prod_{m=1}^{Q}|h_{m}|^{\tfrac{1}{2}+\gamma_{m}}}.
\end{equation}
We can express $G_{\mathcal{D}}(\alpha;\gamma)$ as a convergent Euler product provided that $\mathfrak{R}(\alpha_{k})>0$ and $\mathfrak{R}(\gamma_{m})>0$. Thus,

\begin{equation}
G_{\mathcal{D}}(\alpha;\gamma)=\prod_{\substack{P \ \mathrm{monic} \\ \mathrm{irreducible}}}\left(1+\left(1+\frac{1}{|P|}\right)^{-1}\sum_{0<\sum_{k}a_{k}+\sum_{m}c_{m} \ \mathrm{is \ even}}\frac{\prod_{m=1}^{Q}\mu(P^{c_{m}})}{|P|^{\sum_{k}a_{k}\left(\tfrac{1}{2}+\alpha_{k}\right)+\sum_{m}c_{m}\left(\tfrac{1}{2}+\gamma_{m}\right)}}\right).
\end{equation}

The above expression will enable us to locate the zeros and poles. We obtain

\begin{multline}
\label{eq:6.3.12}
G_{\mathcal{D}}(\alpha;\gamma)=\prod_{\substack{P \ \mathrm{monic} \\ \mathrm{irreducible}}}\Bigg(1+\left(1+\frac{1}{|P|}\right)^{-1}\Bigg[\sum_{\substack{j,k \\ j<k}}\frac{1}{|P|^{\left(\tfrac{1}{2}+\alpha_{j}\right)+\left(\tfrac{1}{2}+\alpha_{k}\right)}}+\sum_{k}\frac{1}{|P|^{1+2\alpha_{k}}}+\\ 
\ \ \ \ \ \ \ \ \ \ \ \ \ \ \ \ \ \ +\sum_{\substack{m<r \\ m,r}}\frac{\mu(P)^{2}}{|P|^{\left(\tfrac{1}{2}+\gamma_{m}\right)+\left(\tfrac{1}{2}+\gamma_{r}\right)}}+\sum_{k}\sum_{m}\frac{\mu(P)}{|P|^{\left(\tfrac{1}{2}+\alpha_{k}\right)+\left(\tfrac{1}{2}+\gamma_{m}\right)}}+\cdots\Bigg]\Bigg),
\end{multline}
where $\cdots$ indicates terms that converge. Remembering that,
\begin{equation}
\zeta_{A}(s)=\prod_{\substack{P \ \mathrm{monic} \\ \mathrm{irreducible}}}\left(1-\frac{1}{|P|^{s}}\right)^{-1}
\end{equation}
and using that
\begin{equation}
\left(1-\frac{1}{|P|^{s}}\right)^{-1}=\sum_{j=0}^{\infty}\left(\frac{1}{|P|^{s}}\right)^{j},
\end{equation}
we have that the terms in \eqref{eq:6.3.12} with $\sum_{k=1}^{K}a_{k}+\sum_{m=1}^{Q}c_{m}=2$ contribute to the zeros and poles. The poles come from terms with $a_{j}=a_{k}=1$, $1\leq j<k\leq K$, and from terms $a_{k}=2$, $1\leq k\leq K$. In addition, there are poles coming from terms with $c_{m}=c_{r}=1$, $1\leq m<r\leq Q$. 

We also note that poles do not arise from terms with $c_{m}=2$ since $\mu(P^{2})=0$. The contribution of zeros arises from terms with $a_{k}=1=c_{m}$ with $1\leq k\leq K$ and $1\leq m\leq Q$. After all this analysis, the contribution, expressed in terms of $\zeta_{A}(s)$, of all these zeros and poles is
\begin{equation}
Y(\alpha;\gamma):=\frac{\prod_{j\leq k\leq K}\zeta_{A}(1+\alpha_{j}+\alpha_{k})\prod_{m<r\leq Q}\zeta_{A}(1+\gamma_{m}+\gamma_{r})}{\prod_{k=1}^{K}\prod_{m=1}^{Q}\zeta_{A}(1+\alpha_{k}+\gamma_{m})}.
\end{equation}

So, when we factor $Y$ out from $G_{\mathcal{D}}$ we are left with the Euler product $A_{\mathcal{D}}$ which is absolutely convergent for all of the variables in small disks around $0$:

\begin{multline}
A_{\mathcal{D}}(\alpha;\gamma)=\prod_{\substack{P \ \mathrm{monic} \\ \mathrm{irreducible}}}\frac{\prod_{j\leq k\leq K}\left(1-\frac{1}{|P|^{1+\alpha_{j}+\alpha_{k}}}\right)\prod_{m<r\leq Q}\left(1-\frac{1}{|P|^{1+\gamma_{m}+\gamma_{r}}}\right)}{\prod_{k=1}^{K}\prod_{m=1}^{Q}\left(1-\frac{1}{|P|^{1+\alpha_{k}+\gamma_{m}}}\right)}\\
\times\left(1+\left(1+\frac{1}{|P|}\right)^{-1}\sum_{0<\sum_{k}a_{k}+\sum_{m}c_{m} \ \mathrm{is \ even}}\frac{\prod_{m=1}^{Q}\mu(P^{c_{m}})}{|P|^{\sum_{k}a_{k}\left(\tfrac{1}{2}+\alpha_{k}\right)+\sum_{m}c_{m}\left(\tfrac{1}{2}+\gamma_{m}\right)}}\right).
\end{multline}

So we can conclude that,
\begin{eqnarray}
& & \sum_{D\in\mathcal{H}_{2g+1,q}}\frac{\prod_{k=1}^{K}Z_{L}(\tfrac{1}{2}+\alpha_{k},\chi_{D})}{\prod_{m=1}^{Q}L(\tfrac{1}{2}+\gamma_{m},\chi_{D})}\nonumber\\
& = &\sum_{D\in\mathcal{H}_{2g+1,q}}\sum_{\epsilon\in\{-1,1\}^{K}}\prod_{k=1}^{K}\mathcal{X}_{D}(\tfrac{1}{2}+\epsilon_{k}\alpha_{k})^{-1/2}Y(\epsilon_{1}\alpha_{1},\ldots,\epsilon_{K}\alpha_{K};\gamma)A_{\mathcal{D}}(\epsilon_{1}\alpha_{1},\ldots,\epsilon_{K}\alpha_{K};\gamma)\nonumber\\
& & +o(|D|),
\end{eqnarray}
using \eqref{eq:Z-function} we have that,

\begin{multline}
\sum_{D\in\mathcal{H}_{2g+1,q}}\frac{\prod_{k=1}^{K}L(\tfrac{1}{2}+\alpha_{k},\chi_{D})}{\prod_{m=1}^{Q}L(\tfrac{1}{2}+\gamma_{m},\chi_{D})}\\
=\sum_{D\in\mathcal{H}_{2g+1,q}}\sum_{\epsilon\in\{-1,1\}^{K}}\prod_{k=1}^{K}\mathcal{X}_{D}(\tfrac{1}{2}+\epsilon_{k}\alpha_{k})^{-1/2}\prod_{k=1}^{K}\mathcal{X}_{D}(\tfrac{1}{2}+\alpha_{k})^{1/2}\\
\times Y(\epsilon_{1}\alpha_{1},\ldots,\epsilon_{K}\alpha_{K};\gamma)A_{\mathcal{D}}(\epsilon_{1}\alpha_{1},\ldots,\epsilon_{K}\alpha_{K};\gamma)+o(|D|),
\end{multline}
moreover,
\begin{equation}
\mathcal{X}_{D}(\tfrac{1}{2}+\epsilon_{k}\alpha_{k})^{-1/2}=|D|^{\tfrac{1}{2}\epsilon_{k}\alpha_{k}}X(\tfrac{1}{2}+\epsilon_{k}\alpha_{k})^{-1/2}
\end{equation}
and
\begin{equation}
\mathcal{X}_{D}(\tfrac{1}{2}+\alpha_{k})^{1/2}=|D|^{-\tfrac{1}{2}\alpha_{k}}X(\tfrac{1}{2}+\alpha_{k})^{1/2},
\end{equation}
and so

\begin{multline}
\prod_{k=1}^{K}\mathcal{X}_{D}(\tfrac{1}{2}+\epsilon_{k}\alpha_{k})^{-1/2}\mathcal{X}_{D}(\tfrac{1}{2}+\alpha_{k})^{1/2}\\
=\prod_{k=1}^{K}|D|^{\tfrac{1}{2}(\epsilon_{k}\alpha_{k}-\alpha_{k})}\prod_{k=1}^{K}X(\tfrac{1}{2}+\epsilon_{k}\alpha_{k})^{-1/2}X(\tfrac{1}{2}+\alpha_{k})^{1/2}\ \ \ \\
=|D|^{\tfrac{1}{2}\sum_{k=1}^{K}(\epsilon_{k}\alpha_{k}-\alpha_{k})}\prod_{k=1}^{K}X(\tfrac{1}{2}+\epsilon_{k}\alpha_{k})^{-1/2}X(\tfrac{1}{2}+\alpha_{k})^{1/2}.
\end{multline}

To put our conjecture in the same form as conjecture 5.2 in \cite{CFZ} and see clearly the analogies between the conjectures for the classical quadratic $L$--functions and the $L$--functions over function fields, we need first to establish the following simple lemma:

\begin{lem}
We have that,
\begin{equation}
X\left(\frac{1}{2}+\epsilon_{k}\alpha_{k}\right)^{-1/2}X\left(\frac{1}{2}+\alpha_{k}\right)^{1/2}=X\left(\frac{1}{2}+\frac{\alpha_{k}-\epsilon_{k}\alpha_{k}}{2}\right).
\end{equation}
\end{lem}
\begin{proof}
Follows directly from the $X(s)=q^{-1/2+s}$.
\end{proof}

If the real parts of $\alpha_{k}$ and $\gamma_{q}$ are positive we are led to
\begin{multline}
\label{eq:6.3.20}
\sum_{D\in\mathcal{H}_{2g+1,q}}\frac{\prod_{k=1}^{K}L(\tfrac{1}{2}+\alpha_{k},\chi_{D})}{\prod_{m=1}^{Q}L(\tfrac{1}{2}+\gamma_{m},\chi_{D})}\\
=\sum_{D\in\mathcal{H}_{2g+1,q}}\sum_{\epsilon\in\{-1,1\}^{K}}|D|^{\tfrac{1}{2}\sum_{k=1}^{K}(\epsilon_{k}\alpha_{k}-\alpha_{k})}\prod_{k=1}^{K}X\left(\frac{1}{2}+\frac{\alpha_{k}-\epsilon_{k}\alpha_{k}}{2}\right)\\
\times Y(\epsilon_{1}\alpha_{1},\ldots,\epsilon_{K}\alpha_{K};\gamma)A_{\mathcal{D}}(\epsilon_{1}\alpha_{1},\ldots,\epsilon_{K}\alpha_{K};\gamma)+o(|D|).
\end{multline}
If we let,
\begin{equation}
H_{\mathcal{D},|D|,\alpha,\gamma}(w)=|D|^{\tfrac{1}{2}\sum_{k=1}^{K}w_{k}}\prod_{k=1}^{K}X\left(\frac{1}{2}+\frac{\alpha_{k}-w_{k}}{2}\right)Y(w_{1},\ldots,w_{K};\gamma)A_{\mathcal{D}}(w_{1},\ldots,w_{K};\gamma)\\
\end{equation}
then the conjecture may be formulated as
\begin{multline}
\label{eq:6.3.21}
\sum_{D\in\mathcal{H}_{2g+1,q}}\frac{\prod_{k=1}^{K}L(\tfrac{1}{2}+\alpha_{k},\chi_{D})}{\prod_{m=1}^{Q}L(\tfrac{1}{2}+\gamma_{m},\chi_{D})}\\
=\sum_{D\in\mathcal{H}_{2g+1,q}}|D|^{-\tfrac{1}{2}\sum_{k=1}^{K}\alpha_{k}}\sum_{\epsilon\in\{-1,1\}^{K}}H_{\mathcal{D},|D|,\alpha,\gamma}(\epsilon_{1}\alpha_{1},\ldots,\epsilon_{K}\alpha_{K};\gamma)+o(|D|),
\end{multline}
which are precisely the formulae given in Conjecture \ref{conj:ratiosconj}.

\begin{rmk}
Note that the formulas \eqref{eq:6.3.20} and \eqref{eq:6.3.21} can be seen as the function field analogues of the formulae (5.27) and (5.29) in \cite{CFZ}.
\end{rmk}

\subsection{Refinements of the Conjecture}

In this section we refine the ratios conjecture first by deriving a closed form expression for the Euler product $A_{\mathcal{D}}(\alpha;\gamma)$, and second by expressing the combinatorial sum as a multiple integral. This is similar to the treatment given in the previous section.

\subsubsection{Closed form expression for $A_{D}$}

Suppose that $\displaystyle{f(x)=1+\sum_{n=1}^{\infty}u_{n}x^{n}}. \ \text{Then}$

\begin{equation}
\sum_{0<n \ \mathrm{is \ even}}u_{n}x^{n}=\frac{1}{2}(f(x)+f(-x)-2)
\end{equation}
and so,
\begin{multline}
\label{eq:6.4.2}
1+\left(1+\frac{1}{|P|}\right)^{-1}\sum_{0<n \ \mathrm{is \ even}}u_{n}x^{n}=1+\left(1+\frac{1}{|P|}\right)^{-1}\left(\frac{1}{2}(f(x)+f(-x)-2)\right)\\
=\frac{1}{1+\frac{1}{|P|}}\left(\frac{f(x)+f(-x)}{2}+\frac{1}{|P|}\right).
\end{multline}

Now, let
\begin{eqnarray}
f\left(\frac{1}{|P|}\right)& = & \sum_{a_{k},c_{m}}\frac{\prod_{m=1}^{Q}\mu(P^{c_{m}})}{|P|^{\sum_{k}a_{k}\left(\tfrac{1}{2}+\alpha_{k}\right)+\sum_{m}c_{m}\left(\tfrac{1}{2}+\gamma_{m}\right)}}\nonumber\\
& = & \sum_{a_{k}}\frac{1}{|P|^{\sum_{k}a_{k}\left(\tfrac{1}{2}+\alpha_{k}\right)}}\sum_{c_{m}}\frac{\prod_{m=1}^{Q}\mu(P^{c_{m}})}{|P|^{\sum_{m}c_{m}\left(\tfrac{1}{2}+\gamma_{m}\right)}}\nonumber\\
& = & \sum_{a_{k}}\prod_{k=1}^{K}\frac{1}{|P|^{a_{k}\left(\tfrac{1}{2}+\alpha_{k}\right)}}\sum_{c_{m}}\prod_{m=1}^{Q}\frac{\mu(P^{c_{m}})}{|P|^{c_{m}\left(\tfrac{1}{2}+\gamma_{m}\right)}}\nonumber\\
& = &
\label{eq:6.4.3} \frac{\prod_{m=1}^{Q}\left(1-\frac{1}{|P|^{1/2+\gamma_{m}}}\right)}{\prod_{k=1}^{K}\left(1-\frac{1}{|P|^{1/2+\alpha_{k}}}\right)}.
\end{eqnarray}

We are ready to prove the following lemma:

\begin{lem}
We have that,
\begin{multline}
1+\left(1+\frac{1}{|P|}\right)^{-1}\sum_{0<\sum_{k}a_{k}+\sum_{m}c_{m} \ \mathrm{is \ even}}\frac{\prod_{m=1}^{Q}\mu(P^{c_{m}})}{|P|^{\sum_{k}a_{k}\left(\tfrac{1}{2}+\alpha_{k}\right)+\sum_{m}c_{m}\left(\tfrac{1}{2}+\gamma_{m}\right)}}\\
= \frac{1}{1+\frac{1}{|P|}}\left(\frac{1}{2}\frac{\prod_{m=1}^{Q}\left(1-\frac{1}{|P|^{1/2+\gamma_{m}}}\right)}{\prod_{k=1}^{K}\left(1-\frac{1}{|P|^{1/2+\alpha_{k}}}\right)}+\frac{1}{2}\frac{\prod_{m=1}^{Q}\left(1+\frac{1}{|P|^{1/2+\gamma_{m}}}\right)}{\prod_{k=1}^{K}\left(1+\frac{1}{|P|^{1/2+\alpha_{k}}}\right)}+\frac{1}{|P|}\right).
\end{multline}
\end{lem}
\begin{proof}
The proof follows directly using \eqref{eq:6.4.2} and \eqref{eq:6.4.3}.
\end{proof}

We have the following corollary to this lemma

\begin{cor}
\begin{multline}
A_{\mathcal{D}}(\alpha;\gamma)=\prod_{\substack{P \ \mathrm{monic} \\ \mathrm{irreducible}}}\frac{\prod_{j\leq k\leq K}\left(1-\frac{1}{|P|^{1+\alpha_{j}+\alpha_{k}}}\right)\prod_{m<r\leq Q}\left(1-\frac{1}{|P|^{1+\gamma_{m}+\gamma_{r}}}\right)}{\prod_{k=1}^{K}\prod_{m=1}^{Q}\left(1-\frac{1}{|P|^{1+\alpha_{k}+\gamma_{m}}}\right)}\\
\times\frac{1}{1+\frac{1}{|P|}}\left(\frac{1}{2}\frac{\prod_{m=1}^{Q}\left(1-\frac{1}{|P|^{1/2+\gamma_{m}}}\right)}{\prod_{k=1}^{K}\left(1-\frac{1}{|P|^{1/2+\alpha_{k}}}\right)}+\frac{1}{2}\frac{\prod_{m=1}^{Q}\left(1+\frac{1}{|P|^{1/2+\gamma_{m}}}\right)}{\prod_{k=1}^{K}\left(1+\frac{1}{|P|^{1/2+\alpha_{k}}}\right)}+\frac{1}{|P|}\right).
\end{multline}
\end{cor}

\subsection{The Final Form of the Ratios Conjecture}

We begin this subsection by quoting the following lemma from \cite{CFZ}.

\begin{lem}
Suppose that
$F(z )=F(z_1,\dots,z_K )$
is a function of $K $ variables,
which is symmetric and regular near $(0,\dots,0)$.
Suppose further that $f(s)$ has a simple pole
of residue~$1$ at $s=0$
but is otherwise analytic in $|s|\le 1$.
Let
either
\begin{equation}
H(z_1,\dots,z_K )=
F(z_1,\dots,z_K)
\prod_{1\le j\le k\le K}
f(z_j+z_k)\end{equation}
or
\begin{equation}
H(z_1,\dots,z_K )=
F(z_1,\dots,z_K)
\prod_{1\le j< k\le K}
f(z_j+z_k).\end{equation}
If $|\alpha_k|<1$ then
\begin{multline}
\sum_{\epsilon \in \{-1,+1\}^K}H(\epsilon_1 \alpha_1,\dots,\epsilon_K \alpha_K)\\
=\frac{(-1)^{K(K-1)/2}2^K}{K!  (2\pi i)^{K }}
\int\limits_{|z_i|=1}\frac{H(z_1, \dots,z_K )
\Delta(z_1^2,\dots,z_{K }^2)^2\prod_{k=1}^K z_k}
{\prod_{j=1}^{K }\prod_{k=1}^{K }
(z_k-\alpha_j)(z_k+\alpha_j)}
\,dz_1\dots dz_{K }
\end{multline}
and
\begin{multline}
\sum_{\epsilon \in \{-1,+1\}^K}\mbox{sgn}(\epsilon)
H(\epsilon_1 \alpha_1,\dots,\epsilon_K \alpha_K)\\ 
=\frac{(-1)^{K(K-1)/2}2^K}{K!  (2\pi i)^{K }}
\int\limits_{|z_i|=1}\frac{H(z_1, \dots,z_K )
\Delta(z_1^2,\dots,z_{K }^2)^2\prod_{k=1}^K \alpha_k}
{\prod_{j=1}^{K }\prod_{k=1}^{K }
(z_k-\alpha_j)(z_k+\alpha_j)}
\,dz_1\dots dz_{K }.
\end{multline}
\end{lem}

%Using this Lemma, we can reformulate Theorem~\ref{thm:Sp} as
%\begin{multline}
%\int_{USp(2N)}\frac{\prod_{k=1}^K\Lambda_A(e^{-\alpha_k})}
%{\prod_{q=1}^Q \Lambda_A(e^{-\gamma_q})}dA
%=e^{-\frac N 2 \sum_{k=1}^K \alpha_k}
%\frac{(-1)^{K(K-1)/2}2^K}{K!  (2\pi i)^{K }}\\
%\times \int\limits_{|z_i|=1}\frac{h_S(z_1, \dots,z_K;\gamma)
%\Delta(z_1^2,\dots,z_{K }^2)^2\prod_{k=1}^K z_k}
%{\prod_{j=1}^{K }\prod_{k=1}^{K }
%(z_k-\alpha_j)(z_k+\alpha_j)}
%\,dz_1\dots dz_{K }.
%\end{multline}

Now we are in a position to present the final form of the ratios conjecture for $L$--functions over functions fields using the contour integrals introduced above. Conjecture \ref{conj:ratiosconj} can be written as follows.

\begin{conj}
Suppose that the real parts of $\alpha_{k}$ and $\gamma_{m}$ are positive. Then
\begin{multline}
\label{eq:6.5.1}
\sum_{D\in\mathcal{H}_{2g+1,q}}\frac{\prod_{k=1}^{K}L(\tfrac{1}{2}+\alpha_{k},\chi_{D})}{\prod_{m=1}^{Q}L(\tfrac{1}{2}+\gamma_{m},\chi_{D})}=\sum_{D\in\mathcal{H}_{2g+1,q}}|D|^{-\tfrac{1}{2}\sum_{k=1}^{K}\alpha_{k}}\frac{(-1)^{K(K-1)/2}2^K}{K!(2\pi i)^{K }}\\
\ \ \ \ \ \ \ \ \ \ \ \ \ \ \ \ \ \times \int\limits_{|z_i|=1}\frac{H_{\mathcal{D},|D|,\alpha,\gamma}(z_1, \dots,z_K;\gamma )
\Delta(z_1^2,\dots,z_{K }^2)^2\prod_{k=1}^K z_k}
{\prod_{j=1}^{K }\prod_{k=1}^{K }
(z_k-\alpha_j)(z_k+\alpha_j)}
\,dz_1\dots dz_{K }
 +o(|D|).
\end{multline}
\end{conj}

\begin{rmk}
If we compare the formula \eqref{eq:6.5.1} with the formula (6.31) presented in \cite{CFZ} we can see clearly the analogy between the classical conjecture and its tranlation for function fields.
\end{rmk}

\section{One--Level Density}

In this section we present an application of the Ratios Conjecture \ref{conj:ratiosconj} for $L$--functions over function fields: we derive a formula the one--level density. The ideas and calculations presented in this section can be seen as a translation to the function field language of the calculations presented in \cite{CS} and \cite{HKS}. 

Our goal is to consider 

\begin{equation}
R_{D}(\alpha;\gamma)=\sum_{D\in\mathcal{H}_{2g+1,q}}\frac{L(\tfrac{1}{2}+\alpha,\chi_{D})}{L(\tfrac{1}{2}+\gamma,\chi_{D})}.
\end{equation}

In this case the conjecture is

\begin{conj}
\label{conj:1-ratio}
With $-\frac{1}{4}<\mathfrak{R}(\alpha)<\frac{1}{4}$, $\frac{1}{\log|D|}\ll\mathfrak{R}(\gamma)<\frac{1}{4}$ and $\mathfrak{I}(\alpha),\mathfrak{I}(\gamma)\ll_{\varepsilon}|D|^{1-\varepsilon}$ for every $\varepsilon>0$, we have
\begin{multline}
R_{D}(\alpha;\gamma)=\sum_{D\in\mathcal{H}_{2g+1,q}}\frac{L(\tfrac{1}{2}+\alpha,\chi_{D})}{L(\tfrac{1}{2}+\gamma,\chi_{D})}\\
=\sum_{D\in\mathcal{H}_{2g+1,q}}\left(\frac{\zeta_{A}(1+2\alpha)}{\zeta_{A}(1+\alpha+\gamma)}A_{D}(\alpha;\gamma)+|D|^{-\alpha}X(\tfrac{1}{2}+\alpha)\frac{\zeta_{A}(1-2\alpha)}{\zeta_{A}(1-\alpha+\gamma)}A_{D}(-\alpha;\gamma)\right)+o(|D|),
\end{multline}
where
\begin{equation}
\label{eq:AD}
A_{D}(\alpha;\gamma)=\prod_{\substack{P \ \mathrm{monic} \\ \mathrm{irreducible}}}\left(1-\frac{1}{|P|^{1+\alpha+\gamma}}\right)^{-1}\left(1-\frac{1}{(|P|+1)|P|^{1+2\alpha}}-\frac{1}{(|P|+1)|P|^{\alpha+\gamma}}\right).
\end{equation}
\end{conj}

To obtain the formula for the one-level density from the ratios conjecture, we note that

\begin{equation}
\sum_{D\in\mathcal{H}_{2g+1,q}}\frac{L^{'}(\tfrac{1}{2}+r,\chi_{D})}{L(\tfrac{1}{2}+r,\chi_{D})}=\frac{d}{d\alpha}R_D(\alpha;\gamma)\bigg|_{\alpha=\gamma=r}.
\end{equation}

Now, a direct calculation gives

\begin{equation}
\frac{d}{d\alpha}
 \frac{\zeta_{A}(1+2\alpha)}{\zeta_{A}(1+\alpha+\gamma)}
A_D(\alpha;\gamma)\bigg|_{\alpha=\gamma=r}=\frac{\zeta_{A}'(1+2r)}{\zeta_{A}(1+2r)}A_D(r;r)+A_D'(r;r)
\end{equation}
and a simple use of the quotient rule give us that
\begin{multline}
\frac{d}{d\alpha}
\left(|D|^{-\alpha}X(\tfrac{1}{2}+\alpha)\frac{\zeta_{A}(1-2\alpha)}{\zeta_{A}(1-\alpha+\gamma)}
A_D(-\alpha;\gamma)\right)\bigg|_{\alpha=\gamma=r}\\
=-(\log q)|D|^{-r}X(\tfrac{1}{2}+r)\zeta_{A}(1-2r)A_{D}(-r;r).
\end{multline}
Also, 
\begin{equation}
A_{D}(r;r)=1,
\end{equation}
\begin{equation}
A_{D}(-r;r)=\prod_{\substack{P \ \mathrm{monic} \\ \mathrm{irreducible}}}\left(1-\frac{1}{|P|}\right)^{-1}\left(1-\frac{1}{(|P|+1)|P|^{1-2r}}-\frac{1}{(|P|+1)}\right)
\end{equation}
and using the logarithmic--derivative formula we can easily obtain that,
\begin{equation}
A_{D}^{'}(r;r)=\sum_{\substack{P \ \mathrm{monic} \\ \mathrm{irreducible}}}\frac{\log|P|}{(|P|^{1+2r}-1)(|P|+1)}.
\end{equation}

Thus, the ratios conjecture implies that the following holds

\begin{thm}
\label{thm:ratio1}
Assuming Conjecture \ref{conj:1-ratio}, $\tfrac{1}{\log|D|}\ll\mathfrak{R}(r)<\frac{1}{4}$ and $\mathfrak{I}(r)\ll_{\varepsilon}|D|^{1-\varepsilon}$ we have
\begin{eqnarray}
& & \sum_{D\in\mathcal{H}_{2g+1,q}}\frac{L^{'}(\tfrac{1}{2}+r,\chi_{D})}{L(\tfrac{1}{2}+r,\chi_{D})}\nonumber\\
& = &\sum_{D\in\mathcal{H}_{2g+1,q}}\left(\frac{\zeta_{A}'(1+2r)}{\zeta_{A}(1+2r)}+A_D'(r;r)-(\log q)|D|^{-r}X(\tfrac{1}{2}+r)\zeta_{A}(1-2r)A_{D}(-r;r)\right)\nonumber\\
& + & o(|D|),
\end{eqnarray}
where $A_{D}(\alpha;\gamma)$ is defined in \eqref{eq:AD}.
\end{thm}

Now we are in a position to derive the formula for the one--level density for the zeros of quadratic Dirichlet $L$--functions over function fields, complete with lower order terms.

Let $\gamma_{D}$ denote the ordinate of a generic zero of $L(s,\chi_{D})$ on the half--line (remember that here, unlike in the number field case, we do not need to assume that all of the complex zeros are on the half--line, because the Riemann hypothesis is established for this family of $L$--functions). As $L(s,\chi_{D})$ is a functions of $q^{-s}$ and so is periodic with period $2\pi i/\log q$ we can confine our analysis of the zeros to $-\pi i/\log q\leq\mathfrak{I}(s)\leq\pi i/\log q$. 

We consider the one--level density

\begin{equation}
S_{1}(f):=\sum_{D\in\mathcal{H}_{2g+1,q}}\sum_{\gamma_{D}}f(\gamma_{D})
\end{equation}
where $f$ is an $(2\pi/\log q)$--periodic even test function and holomorphic.

By Cauchy's theorem we have

\begin{equation}
S_{1}(f)=\sum_{D\in\mathcal{H}_{2g+1,q}}\frac{1}{2\pi i}\left(\int_{(c)}-\int_{(1-c)}\right)\frac{L^{'}(s,\chi_{D})}{L(s,\chi_{D})}f(-i(s-1/2))ds
\end{equation}
where $(c)$ denotes a vertical line from $c-\pi i/\log q$ to $c+\pi i/\log q$ and $3/4>c>1/2+1/\log|D|$. The integral on the $c$--line is

\begin{equation}
\frac{1}{2\pi}\int_{-\pi/\log q}^{\pi/\log q}f(t-i(c-1/2))\sum_{D\in\mathcal{H}_{2g+1,q}}\frac{L^{'}(1/2+(c-1/2+it),\chi_{D})}{L(1/2+(c-1/2+it),\chi_{D})}dt.
\end{equation} 

The sum over $D$ can be replaced by Theorem \ref{thm:ratio1} (see the 1--level density section of \cite{CS} for a more detailed analysis). Next we move the path of integration to $c=1/2$ as the integrand is regular at $t=0$ to obtain

\begin{multline}
\frac{1}{2\pi}\int_{-\pi/\log q}^{\pi/\log q}f(t)\sum_{D\in\mathcal{H}_{2g+1,q}}\Bigg(\frac{\zeta_{A}^{'}(1+2it)}{\zeta_{A}(1+2it)}+A_{D}^{'}(it;it)\\
-(\log q)|D|^{it}X(\tfrac{1}{2}+it)\zeta_{A}(1-2it)A_{D}(-it;it)\Bigg)dt+o(|D|).
\end{multline} 

For the integral on the $1-c$--line, we change variables, letting $s\rightarrow1-s$, and we use the functional equation \eqref{eq:functionalequation} to obtain

\begin{equation}
\frac{L^{'}(1-s,\chi_{D})}{L(1-s,\chi_{D})}=\frac{\mathcal{X}_{D}^{'}(s)}{\mathcal{X}_{D}(s)}-\frac{L^{'}(s,\chi_{D})}{L(s,\chi_{D})},
\end{equation}
where
\begin{equation}
\frac{\mathcal{X}_{D}^{'}(s)}{\mathcal{X}_{D}(s)}=-\log|D|+\frac{X^{'}}{X}(s).
\end{equation}

We thus obtain, finally, the following theorem.

\begin{thm}
Assuming the Ratios Conjecture \ref{conj:1-ratio}, the one--level density for the zeros of the family of quadratic Dirichlet $L$--functions associated with hyperelliptic curves given by the affine equation $C_{D}:y^{2}=D(x)$, where $D\in\mathcal{H}_{2g+1,q}$ is given by
\begin{multline}
S_{1}(f)=\sum_{D\in\mathcal{H}_{2g+1,q}}\sum_{\gamma_{D}}f(\gamma_{D})\\
=\frac{1}{2\pi}\int_{-\pi/\log q}^{\pi/\log q}f(t)\sum_{D\in\mathcal{H}_{2g+1,q}}\Bigg(\log|D|+\frac{X^{'}}{X}(\tfrac{1}{2}-it)+2\Bigg(\frac{\zeta_{A}^{'}(1+2it)}{\zeta_{A}(1+2it)}+A_{D}^{'}(it;it)\\
-(\log q)|D|^{it}X(\tfrac{1}{2}+it)\zeta_{A}(1-2it)A_{D}(-it;it)\Bigg)\Bigg)dt+o(|D|)
\end{multline}
where $\gamma_{D}$ is the ordinate of a generic zero of $L(s,\chi_{D})$ and $f$ is an even and periodic nice test function.
\end{thm}

Defining

\begin{equation}
f(t)=h\left(\frac{t(2g\log q)}{2\pi}\right)
\end{equation}
we now scale the variable $t$ as
\begin{equation}
\tau=\frac{t(2g\log q)}{2\pi}
\end{equation}
and get after a change of variables

\begin{multline}
\sum_{D\in\mathcal{H}_{2g+1,q}}\sum_{\gamma_{D}}h\left(\gamma_{D}\frac{t(2g\log q)}{2\pi}\right)\\
=\frac{1}{2g\log q}\int_{-g}^{g}h(\tau)\sum_{D\in\mathcal{H}_{2g+1,q}}\Bigg(\log|D|+\frac{X^{'}}{X}\left(\tfrac{1}{2}-\tfrac{2\pi i\tau}{2g\log q}\right)+2\Bigg(\frac{\zeta_{A}^{'}(1+\tfrac{4\pi i\tau}{2g\log q})}{\zeta_{A}(1+\tfrac{4\pi i\tau}{2g\log q})}\\
+A_{D}^{'}\left(\tfrac{2\pi i\tau}{2g\log q};\tfrac{2\pi i\tau}{2g\log q}\right)-(\log q)e^{-(2\pi i\tau/2g\log q)\log|D|}\\
\times X\left(\tfrac{1}{2}+\tfrac{2\pi i\tau}{2g\log q}\right)\zeta_{A}\left(1-\tfrac{4\pi i\tau}{2g\log q}\right)A_{D}\left(\tfrac{2\pi i\tau}{2g\log q};\tfrac{2\pi i\tau}{2g\log q}\right)\Bigg)\Bigg)d\tau+o(|D|).
\end{multline}

Writing

\begin{equation}
\zeta_{A}(1+s)=\frac{1/\log q}{s}+\frac{1}{2}+\frac{1}{12}(\log q)s+O(s^{2})
\end{equation}

we have,

\begin{equation}
\frac{\zeta_{A}^{'}(1+s)}{\zeta_{A}(1+s)}=-s^{-1}+\frac{1}{2}\log q-\frac{1}{12}(\log q)^{2}s+O(s^{3}).
\end{equation}

As $g\rightarrow\infty$ only the $\log|D|$ term, the $\zeta_{A}^{'}/\zeta_{A}$ term, and the final term in the integral contribute, yielding the asymptotic

\begin{multline}
\sum_{D\in\mathcal{H}_{2g+1,q}}\sum_{\gamma_{D}}h\left(\gamma_{D}\frac{t(2g\log q)}{2\pi}\right)\\
\sim\frac{1}{2g\log q}\int_{-\infty}^{\infty}h(\tau)\Bigg((\#\mathcal{H}_{2g+1,q})\log|D|-\frac{(\#\mathcal{H}_{2g+1,q})(2g\log q)}{2\pi i\tau}+(\#\mathcal{H}_{2g+1,q})\frac{e^{-2\pi i\tau}}{2\pi i\tau}(2g\log q)\Bigg)d\tau
\end{multline}

But, since $h$ is an even function, the middle term above drops out and the last term can be duplicated with a change of sign of $\tau$, leaving

\begin{equation}
\lim_{g\rightarrow\infty}\frac{1}{\#\mathcal{H}_{2g+1,q}}\sum_{D\in\mathcal{H}_{2g+1,q}}\sum_{\gamma_{D}}h\left(\gamma_{D}\frac{t(2g\log q)}{2\pi}\right)=\int_{-\infty}^{\infty}h(\tau)\left(1-\frac{\sin(2\pi\tau)}{2\pi\tau}\right)d\tau.
\end{equation}

Thus for $q$ fixed and in the large $g$ limit, the one--level density of the scaled zeros has the same form as the one--level density of the eigenvalues of the matrices from $USp(2g)$ chosen with respect to Haar measure and so our result is in agreement with results obtained previously by Rudnick \cite{Ru}.

\section{Acknowledgments}
We would like to thank Brian Conrey, Ze\'{e}v Rudnick and Nina Snaith for very helpful discussions during the course of this project, and the referee for detailed and helpful comments.

\end{document}